\documentclass[openany, amssymb, psamsfonts]{amsart}
\usepackage{mathrsfs,comment}
\usepackage{csquotes}
\usepackage[usenames,dvipsnames]{color}
\usepackage[normalem]{ulem}
\usepackage{url}
\usepackage{multicol}
\usepackage[all,arc,2cell]{xy}
\UseAllTwocells
\usepackage{enumerate}
\usepackage{graphicx}
\usepackage{enumitem}
\usepackage{hyperref}  
\hypersetup{%
  bookmarksnumbered=true,%
  bookmarks=true,%
  colorlinks=true,%
  linkcolor=blue,%
  citecolor=blue,%
  filecolor=blue,%
  menucolor=blue,%
  pagecolor=blue,%
  urlcolor=blue,%
  pdfnewwindow=true,%
  pdfstartview=FitBH}

\def\makeautorefname#1#2{\expandafter\def\csname#1autorefname\endcsname{#2}}
\def\equationautorefname~#1\null{(#1)\null}
\makeautorefname{footnote}{footnote}%
\makeautorefname{item}{item}%
\makeautorefname{figure}{Figure}%
\makeautorefname{table}{Table}%
\makeautorefname{part}{Part}%
\makeautorefname{appendix}{Appendix}%
\makeautorefname{chapter}{Chapter}%
\makeautorefname{section}{Section}%
\makeautorefname{subsection}{Section}%
\makeautorefname{subsubsection}{Section}%
\makeautorefname{theorem}{Theorem}%
\makeautorefname{thm}{Theorem}%
\makeautorefname{cor}{Corollary}%
\makeautorefname{lem}{Lemma}%
\makeautorefname{prop}{Proposition}%
\makeautorefname{pro}{Property}
\makeautorefname{conj}{Conjecture}%
\makeautorefname{defn}{Definition}%
\makeautorefname{notn}{Notation}
\makeautorefname{notns}{Notations}
\makeautorefname{rem}{Remark}%
\makeautorefname{quest}{Question}%
\makeautorefname{exmp}{Example}%
\makeautorefname{ax}{Axiom}%
\makeautorefname{claim}{Claim}%
\makeautorefname{ass}{Assumption}%
\makeautorefname{asss}{Assumptions}%
\makeautorefname{con}{Construction}%
\makeautorefname{prob}{Problem}%
\makeautorefname{warn}{Warning}%
\makeautorefname{obs}{Observation}%
\makeautorefname{conv}{Convention}%

\newtheorem{thm}{Theorem}[section]

\newtheorem{lem}{Lemma}[section]

\theoremstyle{definition}
\newtheorem{defn}{Definition}[section]

\newtheorem{rem}{Remark}[section]
\newtheorem{claim}{Claim}[section]

\makeatletter
\let\c@obs=\c@thm
\let\c@cor=\c@thm
\let\c@prop=\c@thm
\let\c@lem=\c@thm
\let\c@prob=\c@thm
\let\c@con=\c@thm
\let\c@conj=\c@thm
\let\c@defn=\c@thm
\let\c@notn=\c@thm
\let\c@notns=\c@thm
\let\c@exmp=\c@thm
\let\c@ax=\c@thm
\let\c@pro=\c@thm
\let\c@ass=\c@thm
\let\c@warn=\c@thm
\let\c@rem=\c@thm
\let\c@sch=\c@thm
\let\c@equation\c@thm
\numberwithin{equation}{section}
\makeatother

\title{Notes on unknotting algorithms using normal surface theory}

\author{Hakan Solak}

\begin{document}

\begin{abstract}
    These notes present two normal surface theory algorithms to detect the unknot and use the split-link algorithm to prove that the figure-eight knot is knotted.
\end{abstract}

\maketitle

\tableofcontents

\section{Introduction} \label{intro}

These notes transcribe parts of Joel Hass's 2018 lectures at the Henri Poincaré Institute, entitled `Algorithms and complexity in the theory of knots and manifolds.' These parts cover the development of normal surface theory and algorithms for unknotting.

The layout of these notes is as follows. In Section \ref{intro}, a brief history of the unknotting problem is given. In Section \ref{norm}, the theory of normal surfaces is developed after introducing normal curves. In Section \ref{unknotting}, two algorithms to detect the unknot are presented: Haken's original solution and the split link approach. We give particular attention to the split link algorithm and use it in Section \ref{example}, to prove that the figure-eight knot is knotted.

\subsection{Knots}

Knots are primary objects of study in algorithmic topology. For algorithmic questions on knots, it is most natural to be in the PL category, in which knots are polygons in $\mathbb{R}^3$. Polygonal knots can be tabulated in $\mathbb{Z}^3$ as a sequence of points which is easily utilized in software.

Here are some basic questions on knots in the context of algorithmic topology:

\begin{enumerate}
    \item Can we classify knots?
    \item Can we recognize a particular knot such as the unknot?
    \item How hard is it to recognize a knot?
    \item Does topology say something new about complexity classes? (P = NP? NP = coNP?)?
    \item Do undecidable problems arise in the study of knots and 3-manifolds?
\end{enumerate}

The question linking topology and complexity classes is interesting. Currently, little is known about overlaps and non-equivalences between different complexity classes. It is hoped that topological problems can inspire greater understanding, as they often have special properties related to complexity theory. For instance, unknotting lies in both NP and coNP. It is one of the not-so-many problems that lie in both classes (it is not known if it lies in P). Moreover, usually when problems are in NP and coNP, it is fairly easy to establish at least one of these, e.g. graph coloring. For unknotting, however, it is very difficult to show that it is in NP and difficult to show that it is in coNP. Thus, unknotting is quite unusual among algorithmic problems \cite{joel}.

Algorithmic questions about recognizing and classifying knots were instrumental in the early days of computer science and they are central to algorithmic topology. Among these questions, unknotting, i.e. asking for an algorithm to determine whether or not a knot is the unknot, is the simplest. Our intuition in determining if a knot is an unknot is actually not that good. Untangling Haken's Gordian unknot, for instance, is far from intuitive \cite{joel}.

\begin{figure}[h]
    \includegraphics[scale=0.22]{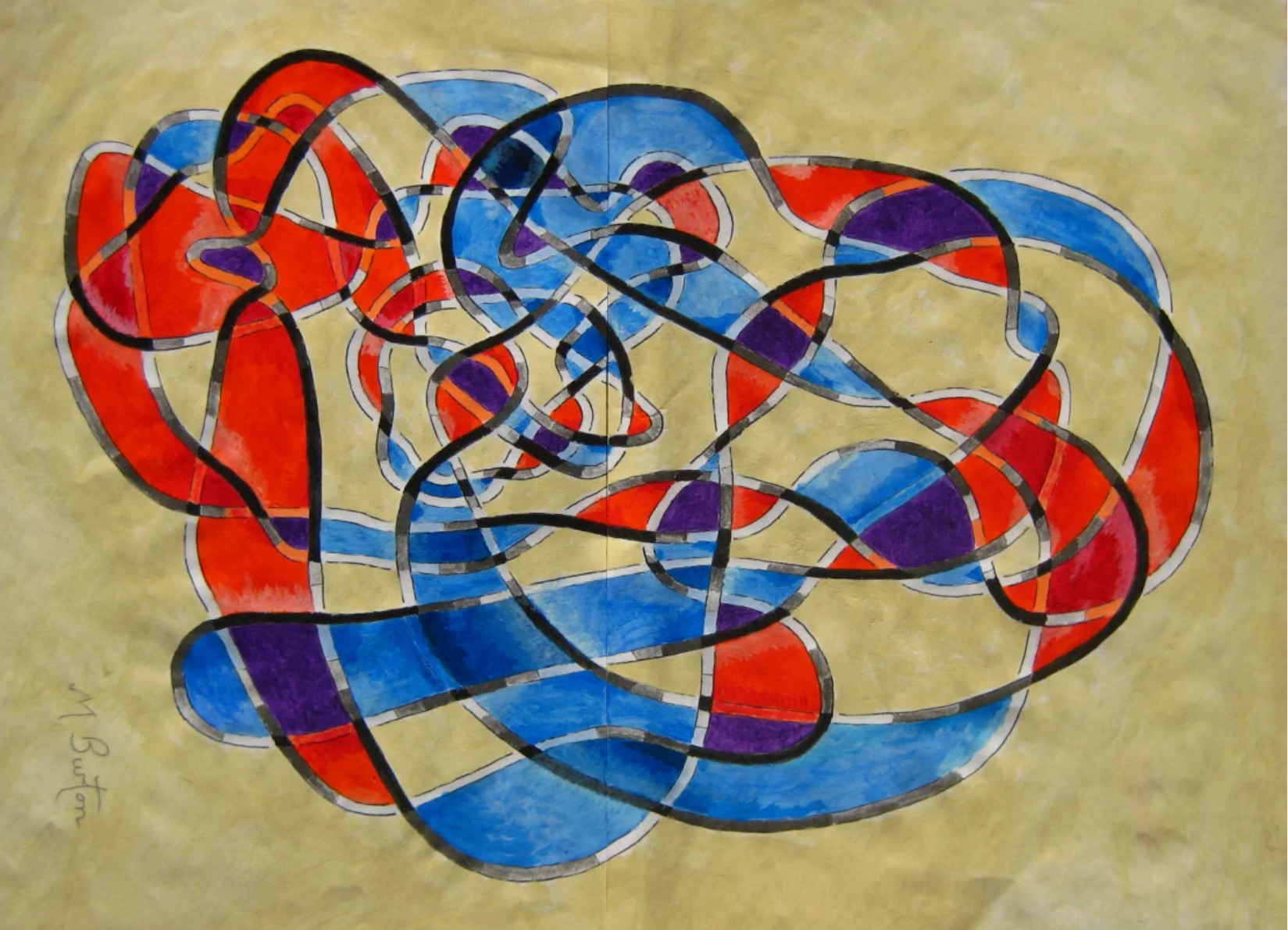}
    \caption{Haken's Gordian unknot. This illustration of the unkot is by Mick Burton. A description of how they drew this can be seen in their blog post at \url{https://mickburton.co.uk/2015/06/05/how-do-you-construct-hakens-gordian-knot/}.}
    \label{fig1}
\end{figure}

The unknotting problem was posed by Max Dehn in 1910, who proposed the search for a procedure to determine if a curve is knotted -- this predates the definition of an algorithm. Dehn had the idea of transforming unknotting into an algortihmic problem in algebra using the following lemma named after him.

\begin{lem} \label{Dehn}
(Dehn, proved by Papakyriakopoulos in 1957) A knot is trivial if and only if its knot group is infinite cyclic.
\end{lem}

This was an appealing transformation as the Wirtinger representation of a knot group can be easily worked out. However, a short time after, Sergei Novikov and William Boone's work showed that a whole class of algebraic problems such as deciding whether a group is isomorphic to the integers cannot admit algorithmic solutions.

Thus, people started looking for different approaches. The first algorithm to detect the unknot was given by Wolfgang Haken in 1961, using normal surface theory and a novel connection to algebra. In fact, Haken's theorem gave algorithmic procedures that solved several problems related to knots in a time when many topological problems were shown to be undecidable \cite{joel}.

\begin{thm} \label{Haken}
(Haken, 1961) There is an algorithmic procedure to
\begin{enumerate}
    \item recognize the unknot,
    \item classify knots,
    \item compute the genus of a knot,
    \item determine if a link is split \cite{haken}.
\end{enumerate}
\end{thm}

These notes present algorithms for (1) and (4).

Since Haken's solution, several algorithms have been given to recognize the unknot. For instance, knot invariants such as Floer homology, A-polynomials, and Khovanov homology distinguish the unknot, so systematically computing them gives algorithms, albeit with worse complexity than the normal surface theory approach.

Not every question about knots is known to admit an algorithmic solution. For example, it is an open problem if there is an algorithm to determine if a knot K has unknotting number at most $n$ \cite{joel}.

\subsection{General manifolds}

Let's briefly mention some algorithmic questions whose objects are general manifolds. There is an idea that algorithmic obstacles are closely related to dimension. For 1,2-manifolds, fast algorithms generally exist. For 3-manifolds, algorithms generally exist, but could have exponential complexity. For manifolds with dimension $\geq 4$, many problems are undecidable.

Here are some famous results:

\begin{thm} \label{Markov}
(Markov, 1958) n-manifold recognition is undecidable for n $\geq 4$.
\end{thm}

\begin{thm} \label{Markov}
(Novikov, 1959) n-sphere recognition is undecidable for n $\geq 5$.
\end{thm}

It is unknown if $S^4$ is recognizable \cite{joel}.

\section{Normal Surface Theory} \label{norm}

Haken's algorithms are based on 3-manifold topology. To determine if a knot is the unknot, he asks if the 3-manifold exterior of the knot in $S^3$ is a solid torus by looking for an embedded disk that spans the knot, which is the condition for being the unknot. Such a disk is called an unknotting disk. Thus, unknotting is the search for an embedded disk. Unknotting disks in $S^3$ can be very complicated even in our PL context, as the following theorem shows:

\begin{thm} \label{H, S, T}
(Hass, Snoeyink, Thurston 2002) There is a sequence of unknotted polygonal curves $K_{n}$ such that $K_{n}$ has less than $11n$ edges and any triangulation of a disk with boundary $K_{n}$ contains at least $2^{n}$ triangular faces \cite{hst}.
\end{thm}

Since unknotting disks can be exponentially complicated, it is essential to have an efficient way to work with surfaces. This is achieved through normal surface theory, which was first studied by Hellmuth Kneser in the 1930's to prove the prime decomposition theorem for 3-manifolds.

Normal surface theory essentially makes it easy for us to work with surfaces. This is very useful, because many problems in 3-manifold topology reduce to understanding surfaces in 3-manifolds. For example, a 3-manifold with infinite fundamental group admits a hyperbolic metric if and only if it contains no essential spheres or tori. Normal surface theory gives algorithms for finding such surfaces and therefore solving these topological questions. These algorithms generally follow two steps:

\begin{enumerate}
    \item \textbf{Normalization:} Start with a class of surfaces, like unknotting disks, and show that if it exists, one can find such a surface that is `normal.'
    \item \textbf{Fundamentalization:} Show that such a normal surface exist among finitely many `fundamental' surfaces.
\end{enumerate}

Thus, normal surface theory limits the search to a finite list, which can be analyzed algorithmically \cite{joel}.

\subsection{Normal curves} Before defining normal surfaces, it is instructive to develop the theory of normal curves. Here's a quote from Joel Hass's lectures on this: `One of the surprising features of many areas of 3 dimensional topology, which I think people don't appreciate much, is how similar they are to the theory of curves on surfaces \cite{joel}.' Recognizing the 3-sphere is an example. Recognizing $S^2$ is easy by computing the Euler characteristic. This method does not translate to three dimensions as all closed 3-manifolds have Euler characteristic zero. However, there is a more complicated algorithm to recognize $S^2$ that generalizes to recognizing $S^3$. In many cases, what to do in 3 dimensions becomes evident if the analog problem for curves on surfaces is well-understood. This applies to normal surface theory: the ideas for normal curves immediately generalize to surfaces \cite{joel}.

\begin{defn}
An arc in a triangle (2-simplex) is called elementary if it is embedded and its two endpoints lie on distinct edges.
\end{defn}

Elementary arcs are classified with respect to the edges they intersect. Here are the three types of elementary arcs in a triangle:

\begin{figure}[h]
    \includegraphics[scale=0.4]{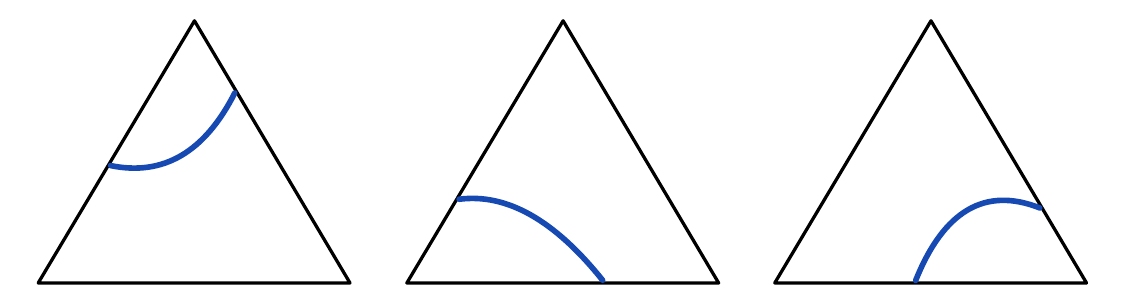}
    \caption{The three types of elementary arcs on a 2-simplex.}
    \label{fig2}
\end{figure}

Now we can define normal curves:

\begin{defn}
A normal curve is an embedded curve on a 2-manifold which is a union of elementary arcs in every 2-simplex.
\end{defn}

A normal curve enters and leaves each simplex at different edges. The process of isotoping an embedded curve to a normal curve is called normalization. Normalization is always possible (recall that we are in the PL category, so no space-filling curves):

\begin{thm} \label{normalization}
An embedded curve on a triangulated surface can be isotoped so that each component is either (1) normal or (2) contained in the interior of a simplex.
\end{thm}

\begin{rem}
In the case of (2), we can just isotope the component to a point and discard it.
\end{rem}

\begin{proof}
First, by transversality, we can always isotope the curve infinitesimally so that it misses the vertices of the triangulation.

The weight of a curve is defined to be the number of times it meets the 1-skeleton of the triangulation. If the curve is not normal or contained in the interior of a simplex, it means that the curve has at at least one arc within a simplex that intersects an edge twice, i.e. not an elementary arc. In this case, we can isotope the curve to decrease its weight by two. The isotopy pushes the arc outside from the edge that it intersects twice, as seen in figure \ref{fig3}.

\begin{figure}[h]
    \includegraphics[scale=0.4]{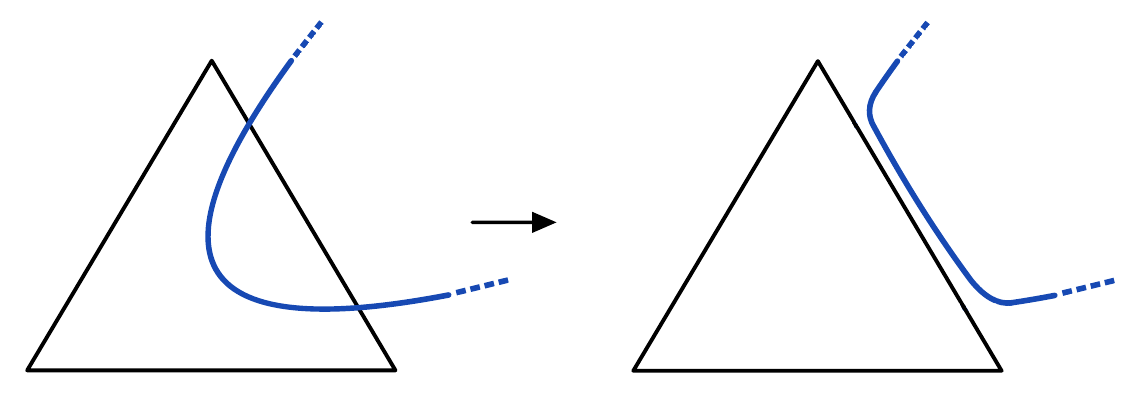}
    \caption{Pushing an arc away from an edge.}
    \label{fig3}
\end{figure}

This must be done starting from an innermost arc for that edge so that no self-intersections are created. Since the weight is finite, this process stops and we get a normal curve or a curve contained in a single simplex \cite{joel}.
\end{proof}

We have also proven the following lemma:

\begin{lem}
A curve of least weight in its isotopy class either lies within a triangle or is normal.
\end{lem}

\subsubsection{Normal curves as vectors} Now that normal curves are defined, it is time to introduce Haken's contribution to normal surface theory. Haken's breakthrough idea was to assign an algebraic object to normal curves and surfaces \cite{haken}. The intersection of a normal curve with a triangle is specified by three numbers each equal to the number of elementary arcs of a given type. Haken assigned to each normal curve a vector in $\mathbb{Z}^{3t}$, where $t$ is the number of triangles in the triangulation. The 3 comes from the fact that each triangle has 3 types of elementary arcs. The vector records the number of elementary arcs of each type in each triangle for the normal curve; see figure \ref{fig4} \cite{joel}.

\begin{figure}[h]
    \includegraphics[scale=0.32]{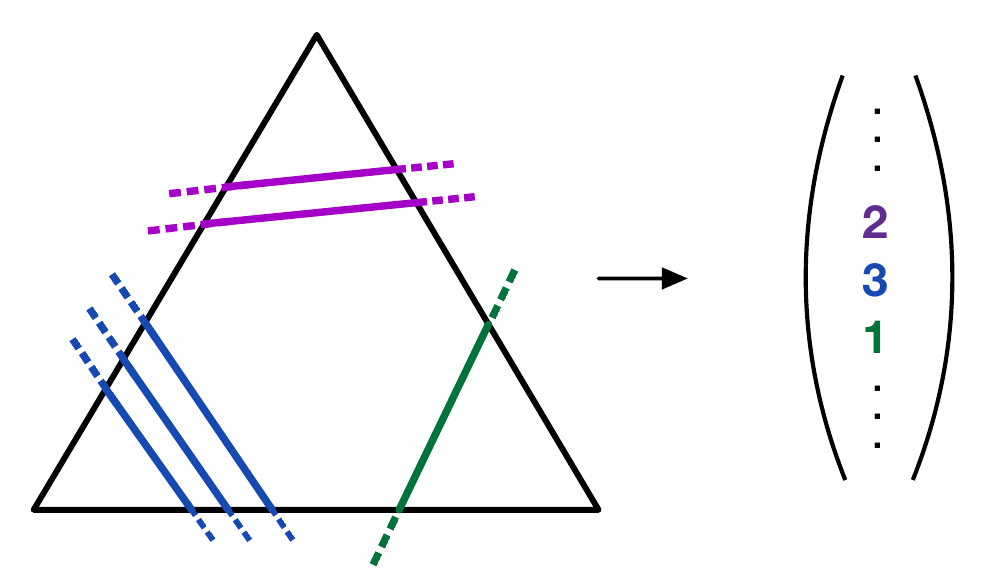}
    \caption{The contribution of a single 2-simplex to the vector representation of a normal curve.}
    \label{fig4}
\end{figure}

This integer vector description is an incredibly efficient way of recording curves on a surface, exponentially more efficient than keeping track of the vertices of the curve as it crosses edges. It is easily implemented in software.

Although every normal curve gives a vector, not every vector gives a normal curve. We get a curve if and only if some simple linear matching equations are satisfied. In particular, the number of elementary arcs that intersect any interior edge must be the same for any triangle that contains that particular edge in the triangulation (note that the curves are properly embedded). In other words, for some orientation, the number of elementary arcs that enter and leave an interior edge must be equal. An example is given in the figure \ref{fig5}.

\begin{figure}[h]
    \includegraphics[scale=0.3]{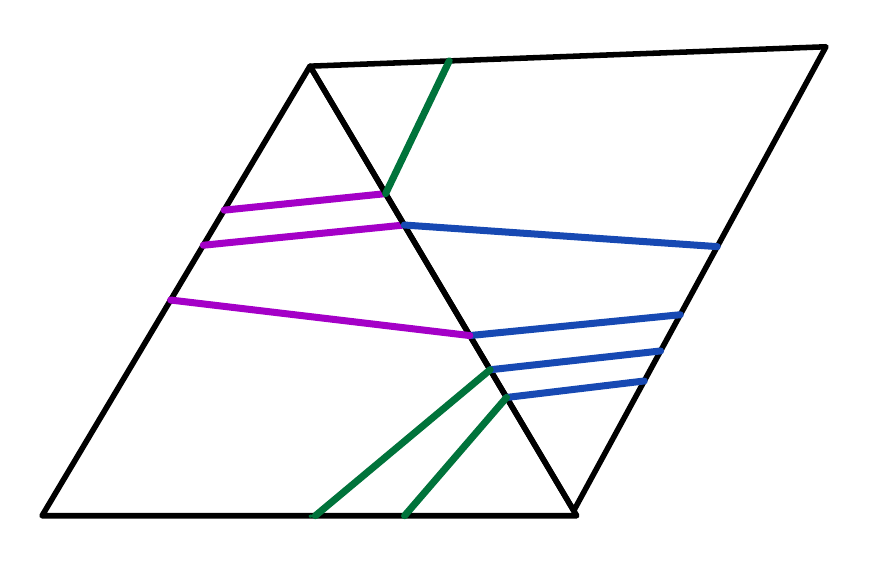}
    \caption{The number of elementary arcs that enter and leave an interior edge must be equal.}
    \label{fig5}
\end{figure}

Also, note that all the integer entries must be greater or equal to zero. This system of linear equations is called the matching equations. It contains $3t$ variables, $3t/2$ equations of the form $v_{i} + v_{j} = v_{k} + v_{l}$, and $3t$ inequalities $v_{i} \geq 0$. Any vector in $\mathbb{Z}^{3t}$ that satisfies the matching equations gives a normal curve \cite{joel} \cite{haken}.

\subsubsection{Haken Sum} A really amazing property of Haken's vector description of normal curves is that vector addition has geometric meaning, i.e. there is an operation to add curves that corresponds to adding the vector representations of the curves. This is called the Haken sum.

Let $a + b$ denote the Haken sum of normal curves $a$ and $b$.

If $a$ and $b$ are disjoint, addition is straightforward. In this case, $a + b = a \cup b$. This means that on each 2-simplex, the image of $a$ on that simplex is superimposed with the image of $b$ to get the image of $a + b$. Thus, the vector representation of $a + b$ is the sum of the vectors of $a$ and $b$. See figure \ref{fig6} for an illustration.

\begin{figure}[h]
    \includegraphics[scale=0.47]{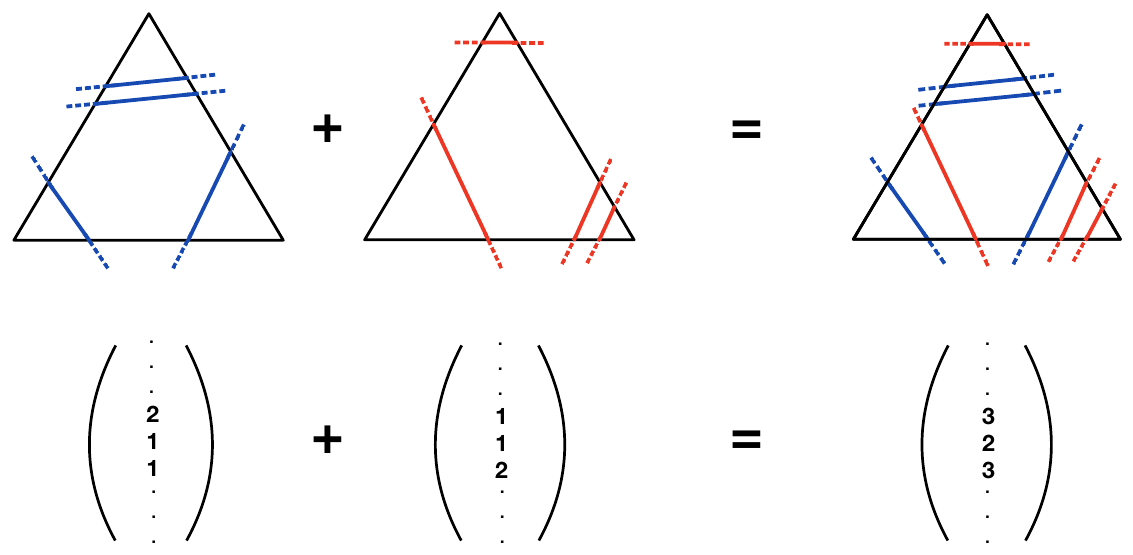}
    \caption{Haken sum of curves that do not intersect in a particular 2-simplex.}
    \label{fig6}
\end{figure}

Suppose $a$ and $b$ intersect. We can slightly deform the curves without changing their vector assignments so that intersections take place in the interiors of 2-simplices. If the elementary arcs of the curves intersect in a simplex, it is necessary to resolve the intersections, since the new curve must be embedded. Resolving intersections involves deleting a very small neighborhood of the intersection and gluing the free ends of arcs together so that there is no intersection; see figure \ref{fig7}. There are two ways to resolve intersections, but only one results in elementary arcs (this can be proven by checking all the cases). This way of resolving is called the regular sum. The other way, called the irregular sum, does not give a normal curve but can be used to reduce weight -- this will be used later \cite{joel}.

\begin{figure}[h]
    \includegraphics[scale=0.4]{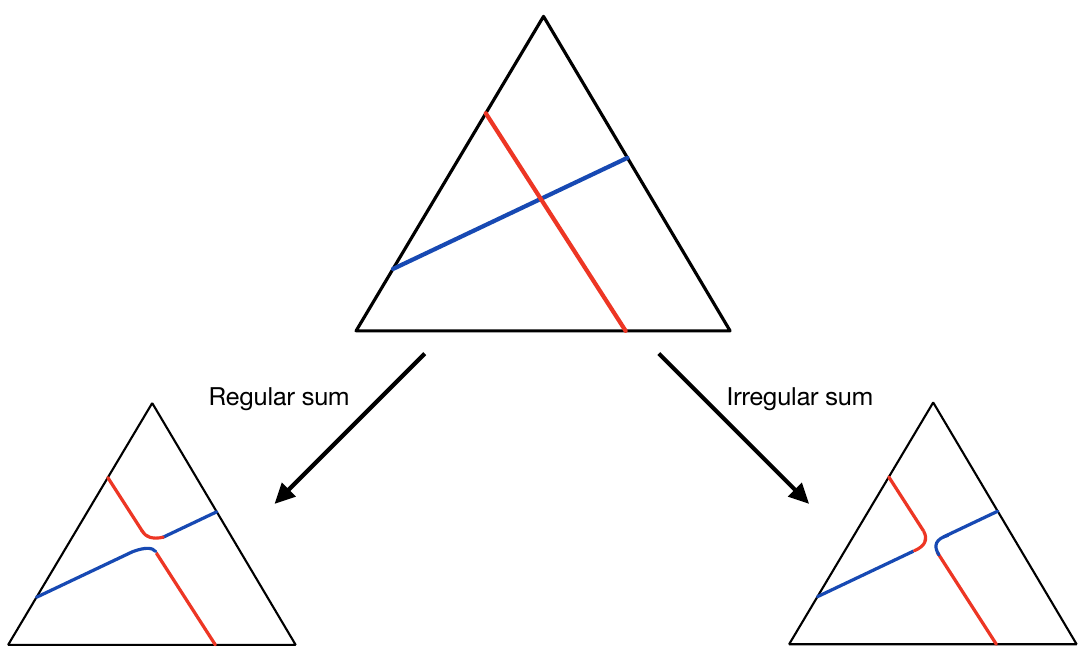}
    \caption{The two ways of resolving intersections, corresponding to a regular sum and an irregular sum.}
    \label{fig7}
\end{figure}

Since there is only one resolution that yields a normal curve, we get a well-defined operation. Once again, the vector of $a + b$ is the sum of the vectors of $a$ and $b$. See figure \ref{fig8} for an illustration.

\begin{figure}[h]
    \includegraphics[scale=0.4]{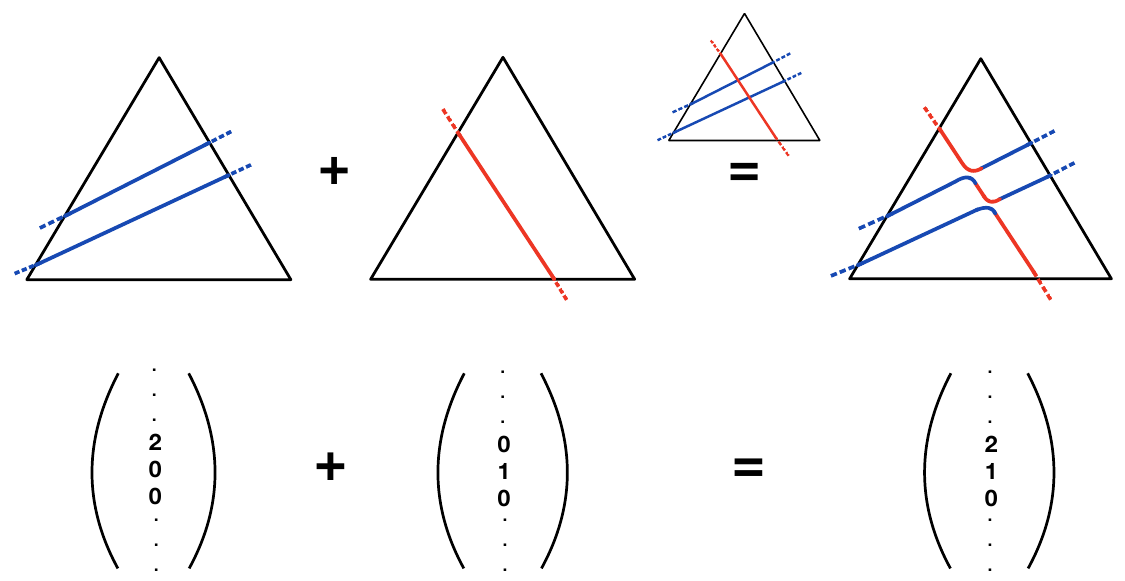}
    \caption{Haken sum of intersecting curves on a 2-simplex.}
    \label{fig8}
\end{figure}

One important observation is that the weight of the Haken sum is equal to the sum of the weights of the summand curves: $w(a + b) = w(a) + w(b)$, as the intersections with the edges of both curves remain.

\begin{rem}
Two normal curves can be isotoped around while preserving their vector assignments and their sum still has the same vector form, since the sum operation only depends on the vectors.
\end{rem}

\subsubsection{Integer Linear Programming}  Using Haken's vector description, the problem of finding normal curves is reduced to the problem of finding solutions to the matching equations. The theory of integer linear programming is used to determine if there is an integer vector solution to the matching equations. A procedure going back to Hilbert constructs a finite set of solutions, called the fundamental solutions, to systems of linear equations with finitely many variables. Any solution to the system is an integer linear combination of these fundamental solutions. The fundamental solutions are called the Hilbert basis for the linear system. Here is a statement of the theorem for the system of matching equations \cite{joel}:

\begin{thm}
There are finitely many fundamental solutions $F_1, ... , F_k \in \mathbb{Z}^{3t}_{+}$ to the matching equations such that any other integer vector solution $w$ is an integer linear combination $w = n_1F_1 + ... + n_kF_k$. These fundamental solutions can be constructed algorithmically.
\end{thm}

Since solutions to matching equations yield normal curves, so do the fundamental solutions. 

\begin{defn}
    The finite number of normal curves that these fundamental solutions correspond to are called fundamental curves.
\end{defn}

Using the Haken sum, we can obtain any normal curve using an integer linear combination of fundamental curves. Therefore fundamental curves are those which cannot be written as the Haken sum of two normal curves. Fundamental curves are the ingredient to obtain an algorithm because if we can show that a solution exists for our problem if and only if a solution exists among the fundamental objects, the problem is solved by inspecting a finite list. A large number of algorithms in topology are of this kind. One issue is that the complexity of finding fundamental solutions is exponential in general \cite{joel}.

\subsubsection{Unknotting in dimension two} The solution given here to unknotting in two dimensions uses the theory of normal curves. This is a warm-up for unknotting in three dimensions, which uses normal surfaces. This algorithm using normal curves is not the simplest solution to the problem, but the ideas generalize to three dimensions.

The knot in this case is an embedded $S^0$, i.e. a pair of disjoint points, $P$ and $Q$, on a triangulated surface, $F$. The algorithm given here solves the special case when both points are on the boundary of the surface, so the surfaces must have nonempty boundary. As in the three dimensional case, a 0-dimensional knot is unknotted if it is the boundary of an embedded disk. So it suffices to determine if the points are on the same component of the surface.

Since any curve can be isotoped to be normal without moving the boundary points and since we are working with boundary points, this is equivalent to asking if there is a normal curve connecting $P$ and $Q$ (note that we are done if the points are on the same edge, so that case is excluded). This, in turn, is equivalent to the existence of an integer vector in $\mathbb{Z}^{3t}$ such that all elementary arcs intersecting the boundary are equal to zero except one of the two elementary arcs that intersects the edge containing $P$ and similarly for $Q$. Thus, we have proven the normalization step of the algorithm: if there is a curve connecting $P$ and $Q$, there must be a normal one as well.

Now we need to verify that if there is a normal arc connecting $P$ and $Q$, there must be a fundamental one connecting them as well. Pick $a$ to be a normal curve of smallest weight that connects $P$ and $Q$. We claim that $a$ is fundamental. Suppose not; then $a = b + c$ for normal curves $b$ and $c$ and one of $b$ or $c$ must have boundary $P$ and $Q$; assume $\partial b = P \cup Q$. Since $w(a) = w(b) + w(c)$, and $w(c) \geq 0$, we get that $w(b) < w(a)$, which is a contradiction. So $a$ is fundamental.

Thus, if two boundary points lie on the same component of a surface, there must be a fundamental curve connecting them. The algorithm, therefore, is checking if there are any fundamental curves connecting the two points \cite{joel}.

\subsection{Normal surfaces} Just as normal curves are defined by how they look on each 2-simplex, normal surfaces in 3-manifolds are defined by how they intersect each 3-simplex.

\begin{defn}
A normal (elementary) triangle is a properly embedded disk in a 3-simplex whose boundary intersects three edges and three faces. A normal (elementary) quadrilateral is a properly embedded disk in a 3-simplex whose boundary intersects four edges and four faces. An elementary disk is a normal triangle or a normal quadrilateral.
\end{defn}

There are four types of normal triangles corresponding to the four vertices of a tetrahedron and three types of normal quadrilaterals corresponding to the three ways that four vertices can be separated into pairs, giving seven types of elementary disks, which are illustrated in figure \ref{fig9}. These are all the ways a flat plane cuts a tetrahedron \cite{joel}.

\begin{figure}[h]
    \includegraphics[scale=0.4]{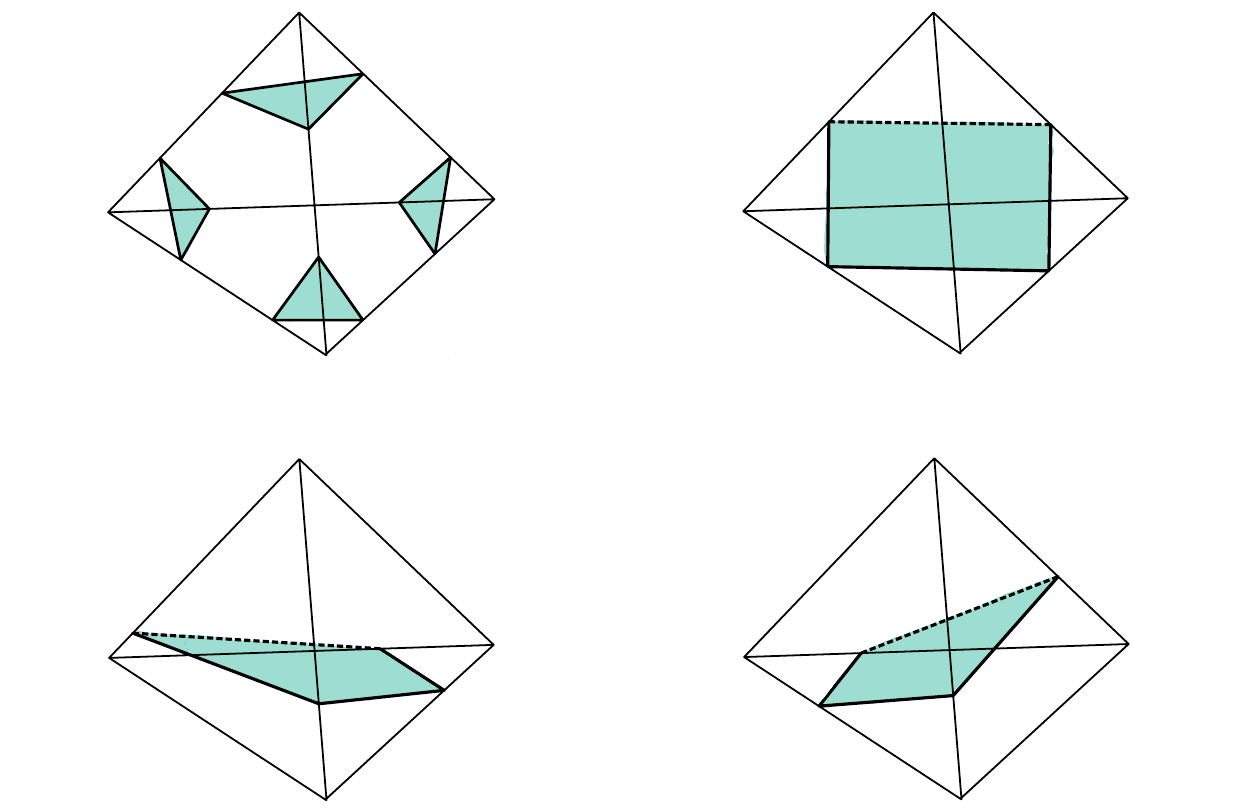}
    \caption{The seven types of elementary disks.}
    \label{fig9}
\end{figure}

\begin{defn}
A normal surface is a properly embedded surface whose intersection with each 3-simplex is a disjoint union of elementary disks.
\end{defn}

While the intersection of a normal surface with a 3-simplex can contain the four types of normal triangles, it can contain only one type of quadrilateral, since the boundary curves of distinct quadrilaterals must intersect.

As with curves, normal surfaces are much simpler than arbitrary surfaces and very easy to work with. Many classes of PL surfaces in a triangulated 3-manifold can be isotoped until normal and all PL surfaces satisfy the following:

\begin{thm} \label{normsurf}
Any properly embedded surface F in a triangulated 3-manifold M can be transformed to a normal surface by a sequence of the following moves:
\begin{enumerate}
    \item isotopy,
    \item compression and boundary compression,
    \item eliminating components lying in a single tetrahedron \cite{joel} \cite{schubert}.
\end{enumerate}
\end{thm}

Compression is cutting a surface at a compression disk (a disk whose intersection with the surface is the boundary $S^1$) and gluing two copies of the compression disk to the two $S^1$ boundaries, as seen in figure \ref{fig10}. It locally splits the surface into two and can cancel 1-handles.

\begin{figure}[h]
    \includegraphics[scale=0.2]{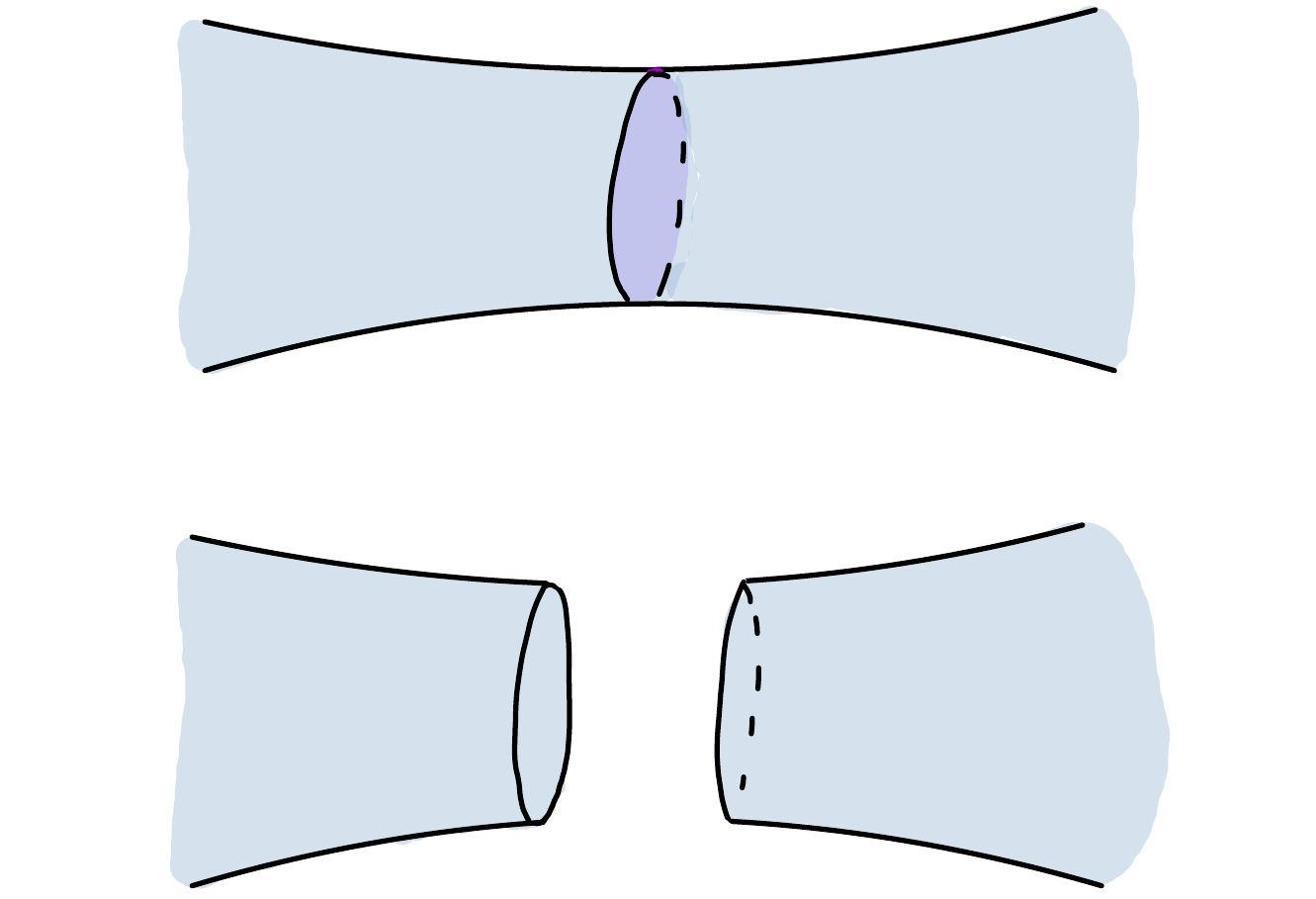}
    \caption{Compression at the purple compression disk.}
    \label{fig10}
\end{figure}

Boundary compression is compression using boundary compressing disks which intersect the surface in a half circle and the boundary of the manifold in the complementary half circle. Figure \ref{fig11} illustrates boundary compression \cite{joel}.

\begin{figure}[h]
    \includegraphics[scale=0.24]{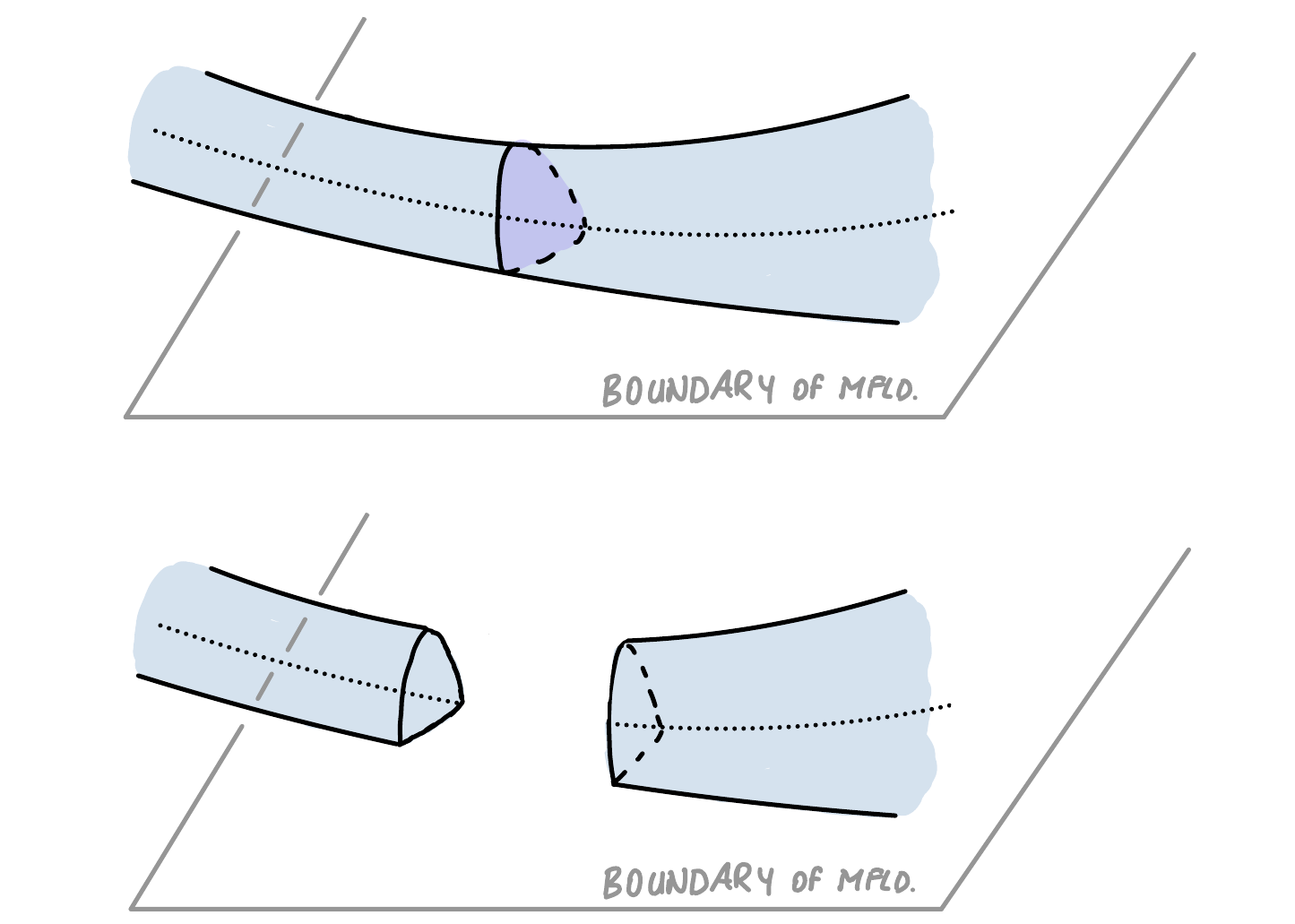}
    \caption{Boundary compression at the purple boundary compressing disk.}
    \label{fig11}
\end{figure}

\begin{proof}

To start the normalization process, fix attention on a single 3-simplex. We first deal with circles of intersection with the faces, as these cannot occur with a normal surface. For each face of the 3-simplex we can remove circles of intersection by performing compressions, beginning from an innermost circle and working our way out, as shown in figure \ref{fig12}.

\begin{figure}[h]
    \includegraphics[scale=0.35]{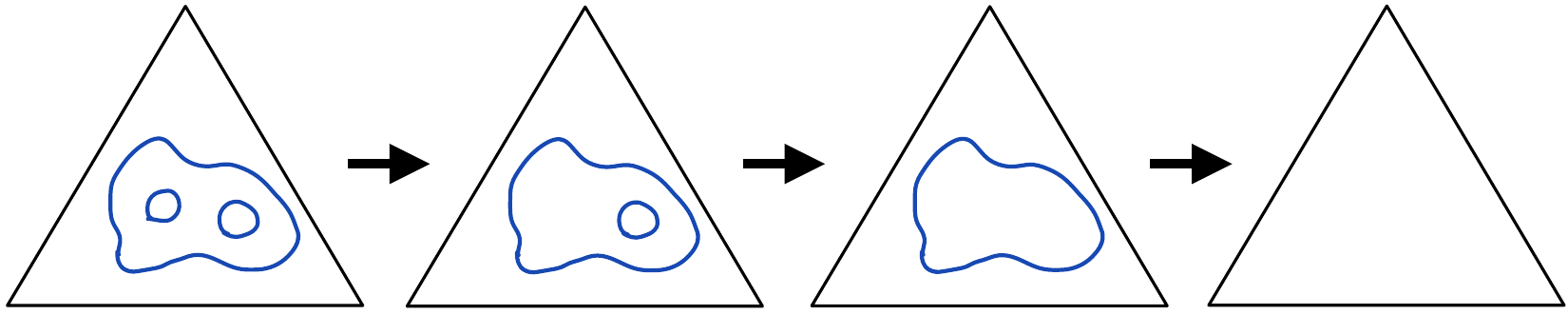}
    \caption{An order for removing circular intersections with faces.}
    \label{fig12}
\end{figure}

If the face is in the interior of the manifold, a compression using the compressing disk lying on the face is sufficient. If the intersection circle is on a boundary face, we can do as in figure \ref{fig13} and discard the component lying entirely in the 3-simplex. These actions reduce the number of intersection curves on the 2-skeleton of M.

\begin{figure}[h]
    \includegraphics[scale=0.37]{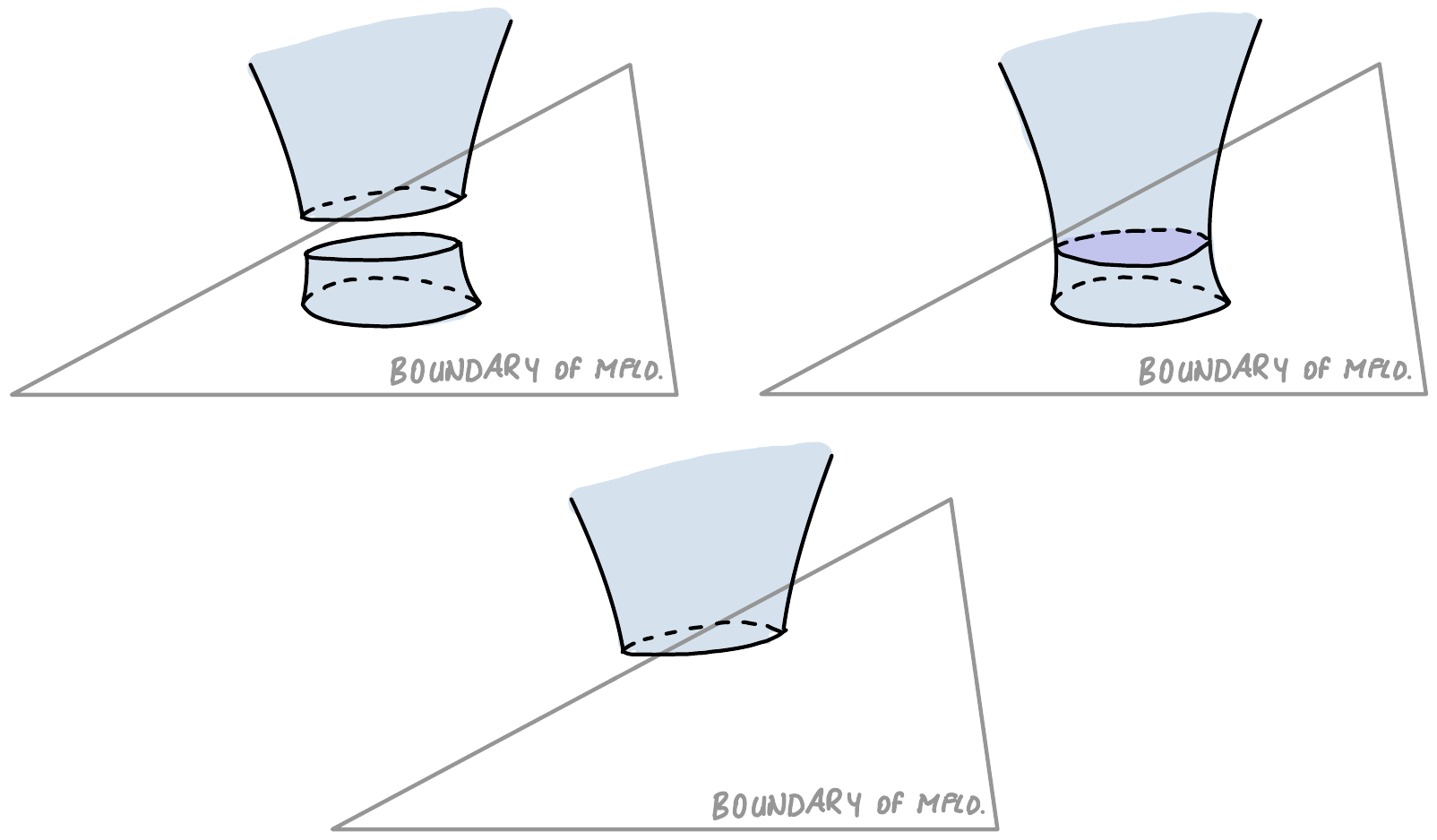}
    \caption{Removing circular intersection curves on a boundary face.}
    \label{fig13}
\end{figure}

We should also deal with arcs on interior faces that intersect a boundary edge twice. For these, we can do a boundary compression, as depicted in figure \ref{fig14}. Again, we start from an innermost one.

\begin{figure}[h]
    \includegraphics[scale=0.35]{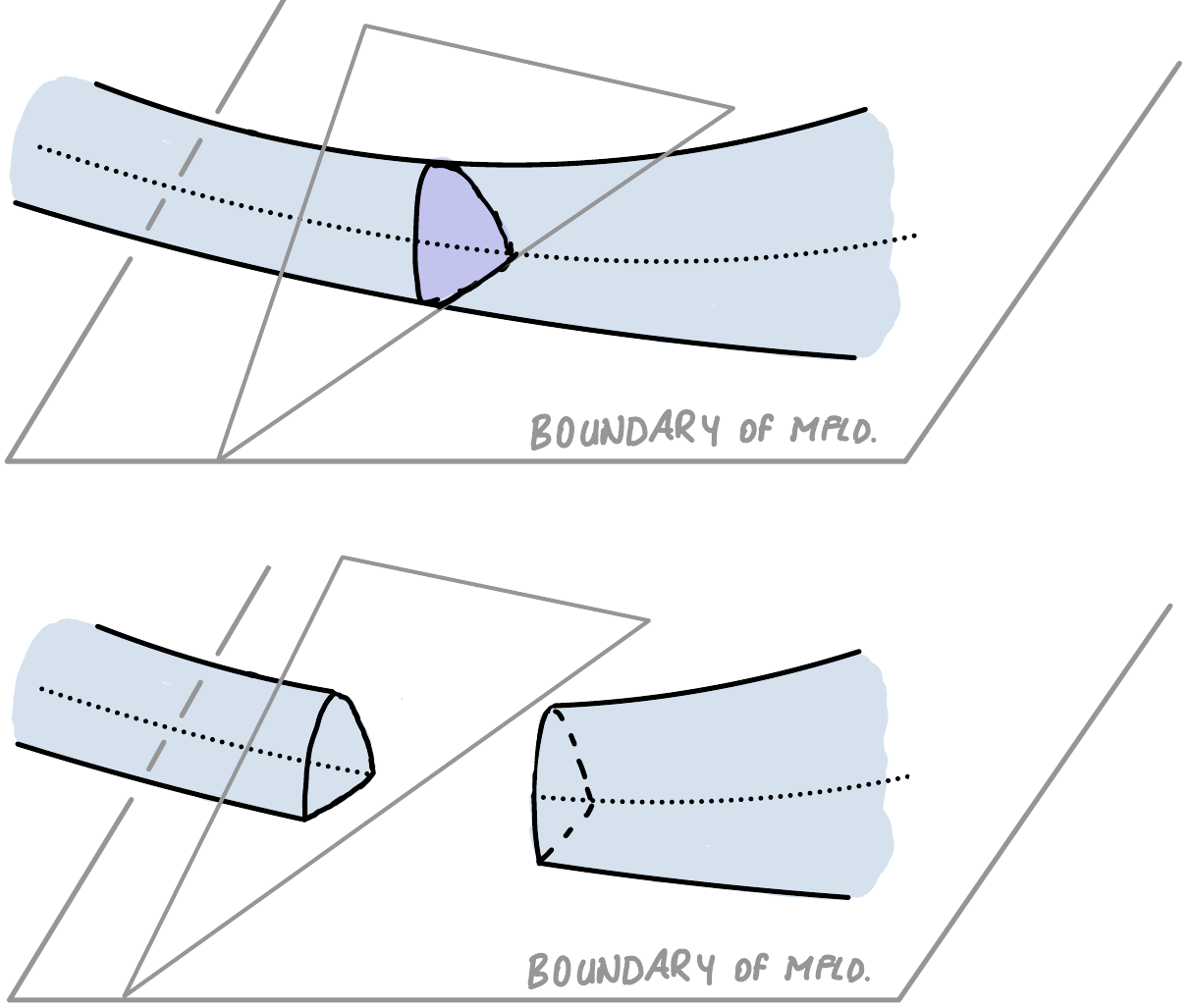}
    \caption{Removing arcs on interior faces that intersect a boundary edge twice.}
    \label{fig14}
\end{figure}

For any other arc on a face that intersects an edge twice and is not of the type just handled, we can isotope the surface across the edge to get rid of them. Once again we start from an innermost one. See figure \ref{fig15} for a illustration. This move decreases weight by two (when we push across, loops may be created on some other face, but as explained at the end of this proof, simple complexity guarantees that progress is being made).

\begin{figure}[h]
    \includegraphics[scale=0.3]{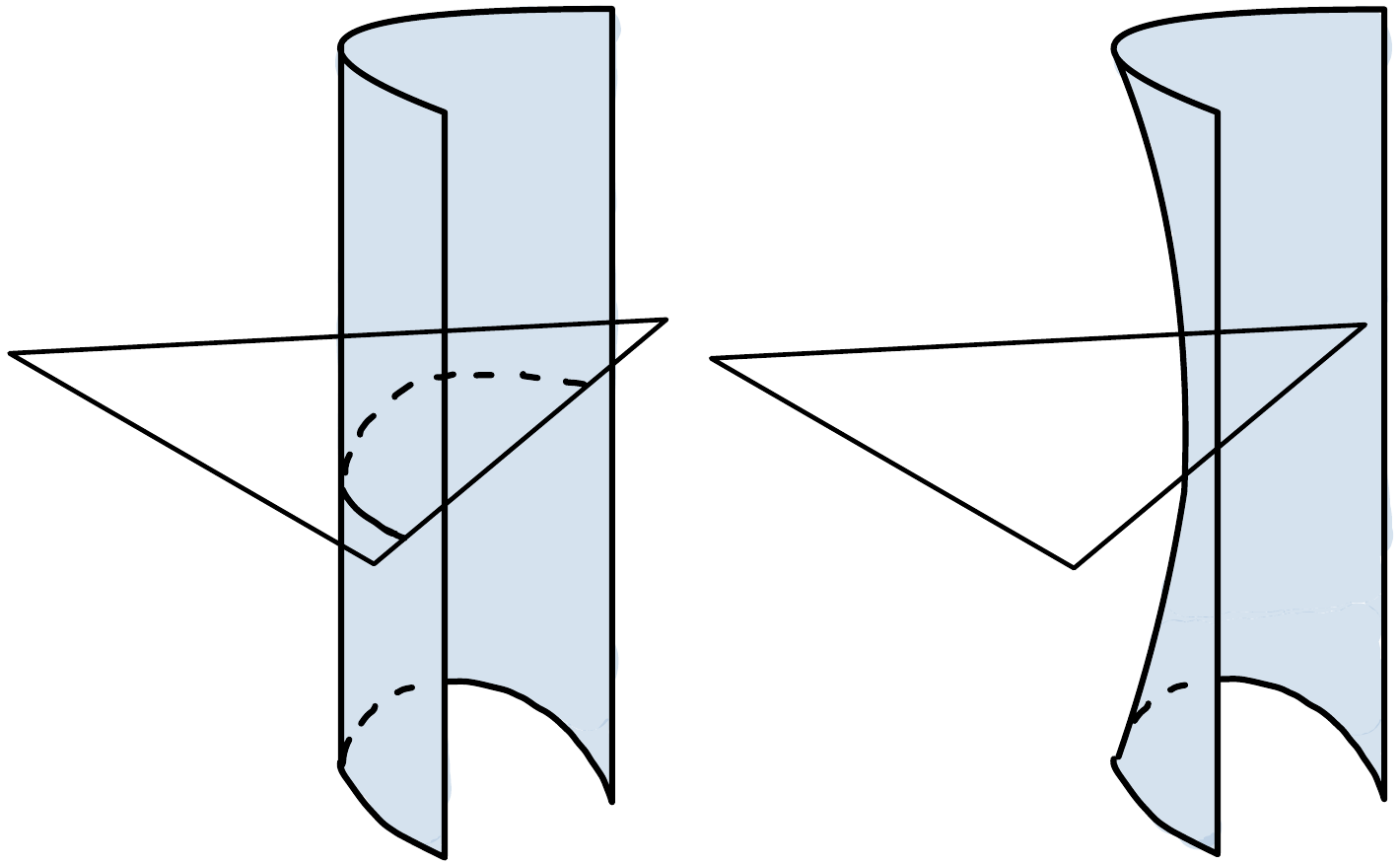}
    \caption{Getting rid of arcs that intersect an edge twice by isotoping away from the face.}
    \label{fig15}
\end{figure}

With these moves, all intersection curves on the faces become elementary arcs. However, there still can be some non-normal disks that cannot be identified just by looking at the intersections with faces. These disks are defined by the property that they must intersect an edge of a 3-simplex in more than one point. In this case, we can isotope across this edge, starting from an innermost such disk, and reduce the weight by two. By examining the cases we can be convinced that this can always be done; see figure \ref{fig16} for an example \cite{joel}.

\begin{figure}[h]
    \includegraphics[scale=0.33]{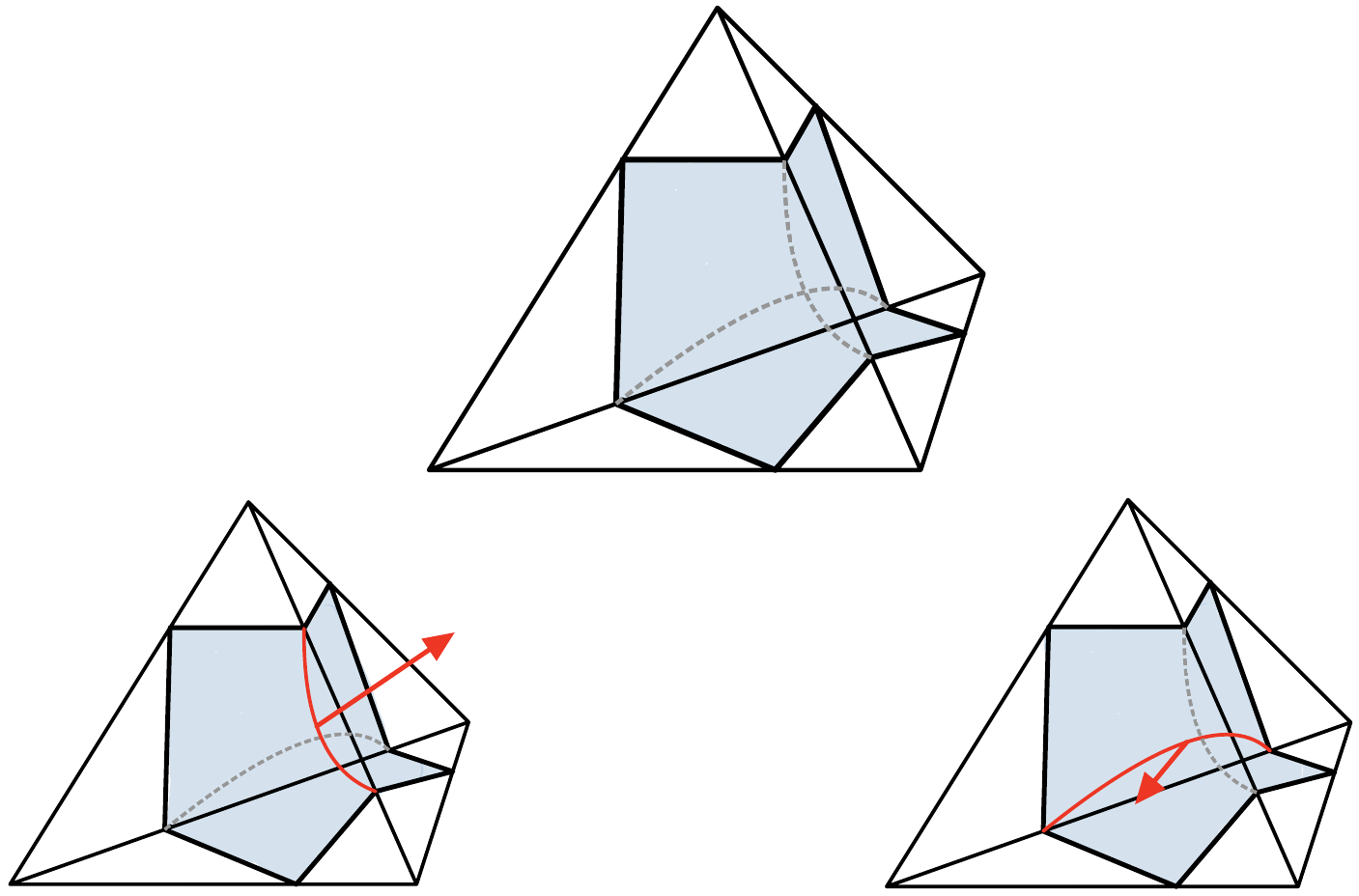}
    \caption{An example of handling non-elementary disks that have correct intersections with the faces.}
    \label{fig16}
\end{figure}

Each of the moves described here reduce the pair (weight, number of intersection curves in the 2-skeleton). Thus, we always make progress. A connected surface in a tetrahedron that is not compressible, not boundary compressible, and which intersects any edge of the tetrahedron in at most one point is an elementary disk or contained entirely in the interior. Thus, we end up with a normal surface, possibly empty \cite{joel}.

\end{proof}

\begin{rem}
This process does not increase genus. If we start with a disk, we will always get components that are disks.
\end{rem}

\begin{rem}
The normal surface obtained at the end is not unique; it depends on the order of the moves.
\end{rem}

\begin{rem}
For surfaces of interest in algorithmic topology such as splitting spheres and unknotting disks, we want normalization to preserve their defining property, so that we can limit our search to normal surfaces.
\end{rem}

Now let's outline Haken's connection of normal surface theory to algebra, analogous to the corresponding connection for normal curves. Once again, we can assign vectors to each normal surface that count elementary disks in each 3-simplex. This vector is an element of $\mathbb{Z}^{7t}$, since each 3-simplex contributes seven coordinates, corresponding to the seven types of elementary disks; see figure \ref{fig17}.

\begin{figure}[h]
    \includegraphics[scale=0.35]{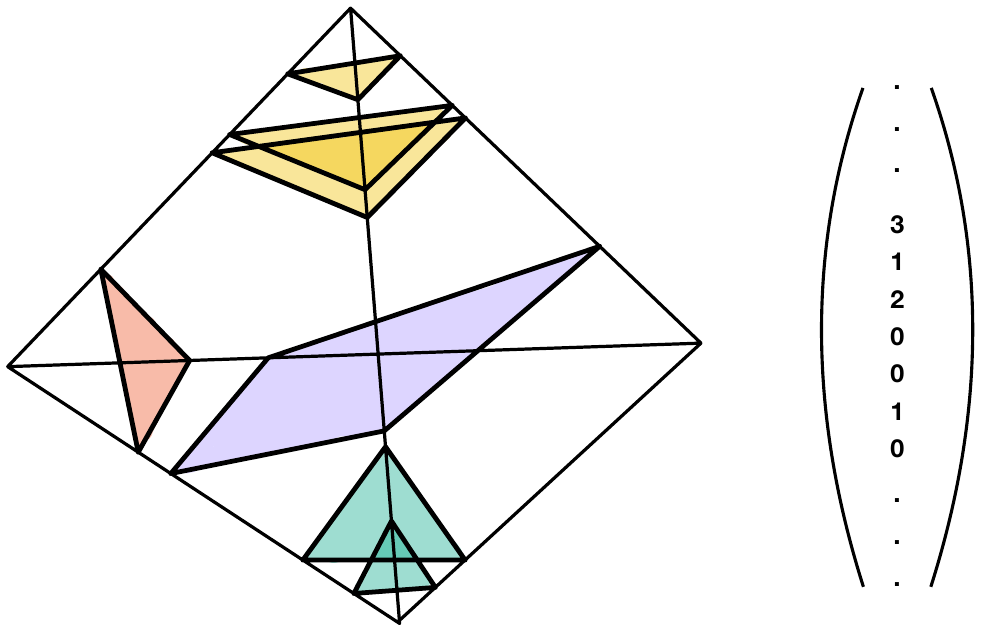}
    \caption{The contribution of a single 3-simplex to the vector representation of a normal surface.}
    \label{fig17}
\end{figure}

Again, a vector in $\mathbb{Z}^{7t}$ describes a normal surface if its entries satisfy the matching equations for normal surfaces, which similarly ensure correct behavior at the faces of the 3-simplices. Matching equations are illustrated in figure \ref{fig18}. The linear system contains $6t$ equations (each face contributes 3 equations) of the form $v_i + v_j = v_k + v_l$ and $7t$ inequalities of the form $v_i \geq 0$ \cite{joel}.

\begin{figure}[h]
    \includegraphics[scale=0.35]{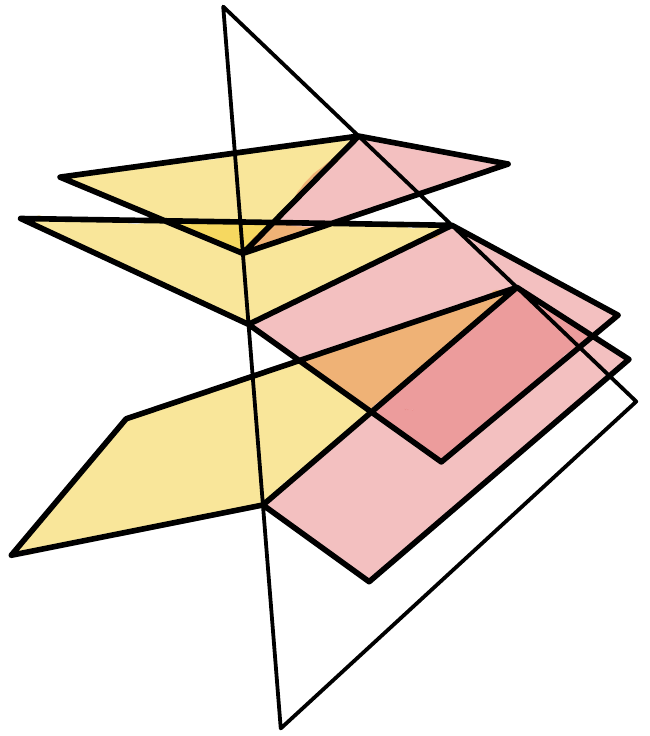}
    \caption{Suitable boundary behavior of elementary disks at the faces. Matching equations guarantee such behavior.}
    \label{fig18}
\end{figure}

However, solving the matching equations for normal surfaces requires additional bookkeeping arising from the obstruction to having distinct elementary quadrilaterals in a 3-simplex. Figure \ref{fig19} gives an example to show that two different types of elementary quadrilaterals in a 3-simplex cannot be disjoint.

\begin{figure}[h]
    \includegraphics[scale=0.28]{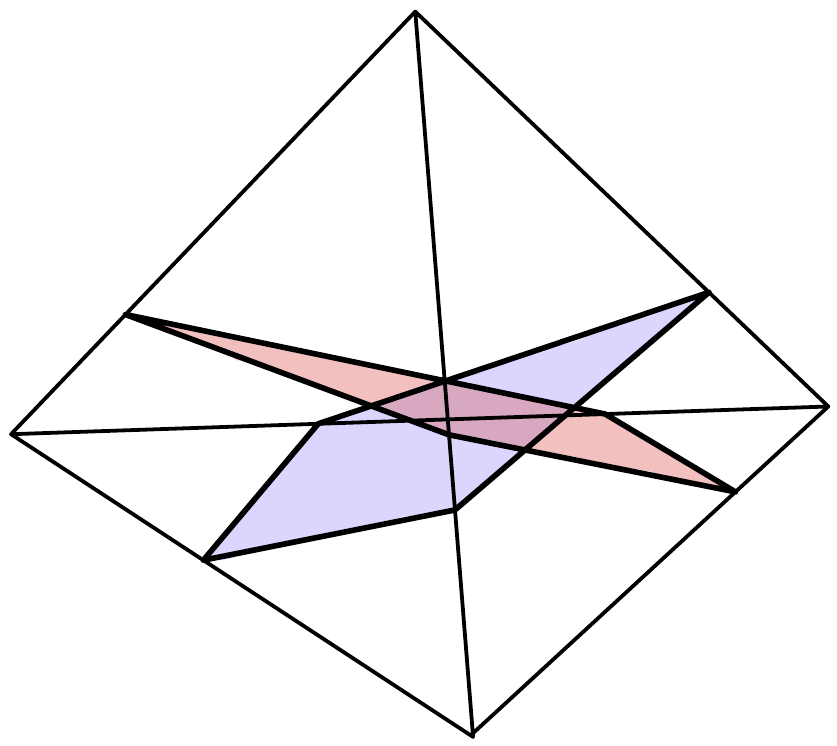}
    \caption{Two different types of elementary quadrilaterals intersect.}
    \label{fig19}
\end{figure}

The vectors that describe normal surfaces must satisfy the condition that only one of the three coordinates that count elementary quadrilaterals can be nonzero for each 3-simplex. Therefore, the strategy is to consider $3^{t}$ special cases of the matching equations, amounting to the number of ways one can choose a quadrilateral type for each 3-simplex. So there is exponential time complexity tied to choosing quadrilaterals. In fact, all other aspects of the unknotting algorithm have polynomial time complexity. The quadrilaterals are a significant source of irritation for matters of efficiency \cite{joel}.

As before, adding integer vectors associated with normal surfaces corresponds to a Haken sum of the normal surfaces. The definition of the Haken sum is analogous to the 2-dimensional case. In this case, curves of intersection between elementary disks of different normal surfaces must be resolved. Once again, there are two ways to resolve intersections and only one, called the regular sum, that gives a normal surface (this can be verified by checking all cases). Performing the other resolution at intersections defines another operation called the irregular sum. This gives a non-normal surface which can be isotoped to become normal. Note that this surface has smaller weight. Resolving intersections is done by cutting the intersecting elementary disks along the intersection curve and gluing pieces of one disk to the piece of another at the intersection curve. Figure \ref{fig20} shows the regular (Haken) and irregular sums \cite{joel}.

\begin{figure}[h]
    \includegraphics[scale=0.44]{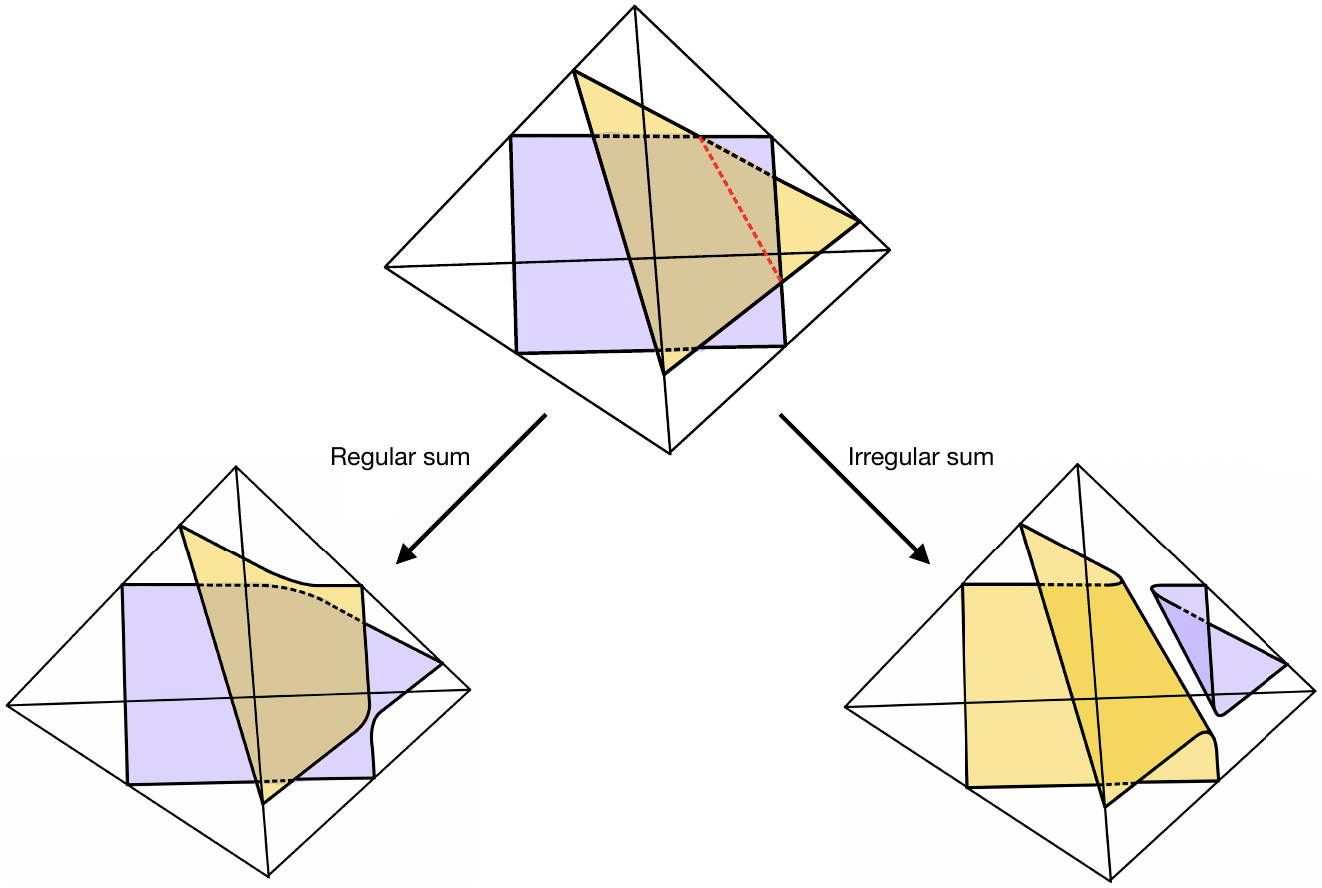}
    \caption{The two ways of resolving an intersection curve between elementary disks, corresponding to a regular and an irregular sum.}
    \label{fig20}
\end{figure}

While taking a Haken sum, there can only be one type of normal quadrilateral in each 3-simplex. This makes sense intuitively because if the two normal surfaces had two different types of normal quadrilaterals in a 3-simplex, this would imply that the Haken sum must also have these two types, which is not possible, since the Haken sum is a normal surface. Moreover, there is no way to take a Haken sum of two different types of quadrilaterals and get elementary disks. A lot of people have made efforts to work around this to get an immersed normal surface theory, but it has not been achieved yet. An immersed normal surface theory is very desirable: if we can use it the way we use embedded normal surface theory, we could, for instance, solve the word problem for 3-manifolds \cite{joel}.

Some nice properties hold under the Haken sum $A = B + C$:
\begin{enumerate}
    \item Euler characteristic is additive: $\chi(A) = \chi(B) + \chi(C)$. This holds because the Haken sum does not change the total number of vertices, edges, and faces among the surfaces.
    \item Weight is also additive: $w(A) = w(B) + w(C)$. This is because all intersections with the 1-skeleton are preserved.
\end{enumerate}

\section{Unknotting} \label{unknotting}

This section presents two approaches to unknotting: the unknotting disk approach and the split link approach. In the unknotting disk approach, which is the solution that Haken originally gave for unknotting, an algorithm is given to determine if an unknotting disk exists. In the split link approach, an algorithm is given to determine if there is a 2-sphere that separates a two component link. This solves unknotting when applied to a knot and its 0-pushoff:

\begin{defn}
    A pushoff of a knot is a parallel copy of the knot, i.e. the knot and its 0-pushoff form the boundary of an annulus. The 0-pushoff has linking number 0 with the knot.
\end{defn}

For our purposes, it suffices to know that the linking number of the pushoff is 0 if the pushoff represents the trivial homology class of the knot exterior. For an in-depth discussion of the linking number, see \cite{rolfsen}.

As before, the algorithms have the following scheme:
\begin{enumerate}
    \item Fix a triangulation of the 3-manifold.
    \item Construct the fundamental normal surfaces in this triangulation.
    \item Check if one of these has the sought after property, i.e. if it is an unknotting disk or a separating sphere \cite{joel}.
\end{enumerate}

\subsection{Unknotting disk approach}

Haken's approach to unknotting is to search for an embedded disk in the complement of the knot in $S^3$ that spans the knot. This is equivalent to showing that the boundary torus of $S^3 - R$, where $R$ is a tubular neighborhood of the knot, is compressible, i.e. there exists a disk whose boundary traverses it longitudinally.

\subsubsection{The triangulation}
The setup for Haken's unknotting disk approach is to first have a triangulation of $S^3$ with the knot embedded in the 1-skeleton. To obtain this triangulation, one can start with a big tetrahedron containing the knot. Subdivide it to get the knot to lie in the 1-skeleton. $S^3$ can then be triangulated by having another tetrahedron outside of this one and identifying the faces. By barycentric subdivision, or some other method, remove a tubular neighborhood of the knot, thereby getting the triangulation of a 3-manifold with torus boundary.

Unlike in the split link approach, as we will soon see, it is necessary to triangulate $S^3 - R$ for this algorithm, since if we only triangulate $S^3$ with the knot in its 1-skeleton, there cannot be a normal surface that spans the knot, because normal surfaces intersect edges at a discrete set \cite{joel}.

There are a lot efficient algorithms for constructing such a triangulation, which is called a good triangulation in \cite{hlp}.

When studying the computational complexity associated with this method, questions such as the following become important:

\begin{enumerate}
    \item What is the minimum number of straight edges needed to realize a knot?
    \item Starting from a knot with $n$ crossings, how many tetrahedra are needed to form a simplicial complex with the knot embedded in its 1-skeleton? 
\end{enumerate}

The following theorem addresses question (2):

\begin{thm} \label{hlp}
(Hass, Lagarias, Pippenger, 1999) Given a knot, $K$, with a diagram, $D$, with $n$ crossings, one can construct in time $O(nlogn)$ a combinatorial triangulation of $S^3$ using at most $253,440(n+1)$ tetrahedra, which contains a good triangulation of $S^3 - K \cong S^3 - R(K)$ \cite{hlp}.
\end{thm}

Here, $R(K)$ is a regular neighborhood of $K$. A combinatorial triangulation is a triangulation for which the link of every face is a sphere. With Hass, Lagarias, and Pippenger's construction, a triangulation of the trefoil complement had 750,000 tetrahedra \cite{hlp}.

\subsubsection{Boundary conditions for unknotting disk}
For the unknotting disk approach to unknotting we will skip the proof for normalization and fundamentalization and just give the algorithm. These are presented in detail in \cite{Kent}.

Thus, assuming that a boundary compressing disk exists, we accept without proof that a boundary compressing disk exists among the fundamental surfaces. The conditions to look for to find a compressing disk among the fundamental surfaces are the following:
\begin{enumerate}
    \item The surface is a disk: check using the Euler characteristic.
    \item The boundary of the disk is a longitude of $\partial (S^3 - R)$: search for fundamental surfaces with all variables corresponding to elementary disks that meet $\partial (S^3 - R)$ being zero, except those along a longitude \cite{joel}.
\end{enumerate}

If a fundamental surface exists with these properties, the knot is unknotted; otherwise it is knotted.

\subsection{Split link approach}
The premise of the split link approach to unknotting is the following observation:

\begin{thm}
A knot is unknotted if and only if the link consisting of the knot and its 0-pushoff is split.
\end{thm}

\begin{proof}
$(\Longrightarrow)$ Suppose that our knot is the unknot. Then the Seifert surface of the knot, i.e. a surface whose boundary is the knot, is an unknotting disk. The 0-pushoff of any knot can be obtained by pushing the knot up or down along the normal vector field of its Seifert surface \cite{rolfsen}; see figure \ref{fig21}.

\begin{figure}[h]
    \includegraphics[scale=0.35]{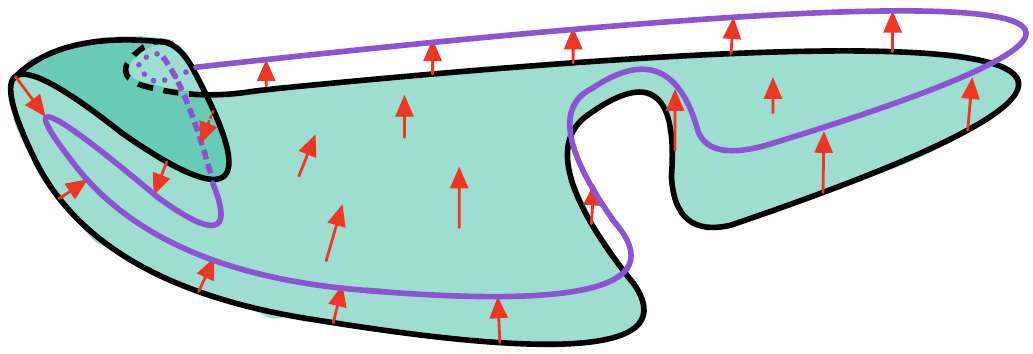}
    \caption{Pushing the knot along the normal vector field of its Seifert surface, given by red arrows, yields a 0-pushoff, illustrated in purple.}
    \label{fig21}
\end{figure}

We can obtain a sphere that separated the knot and its 0-pushoff by pushing the interior of the Seifert surface of the original knot in the positive and negative directions; see figure \ref{fig22} for an illustration. Thus, we get a sphere whose equator is the original knot and with respect to which the 0-pushoff lies in the exterior.

\begin{figure}[h]
    \includegraphics[scale=0.37]{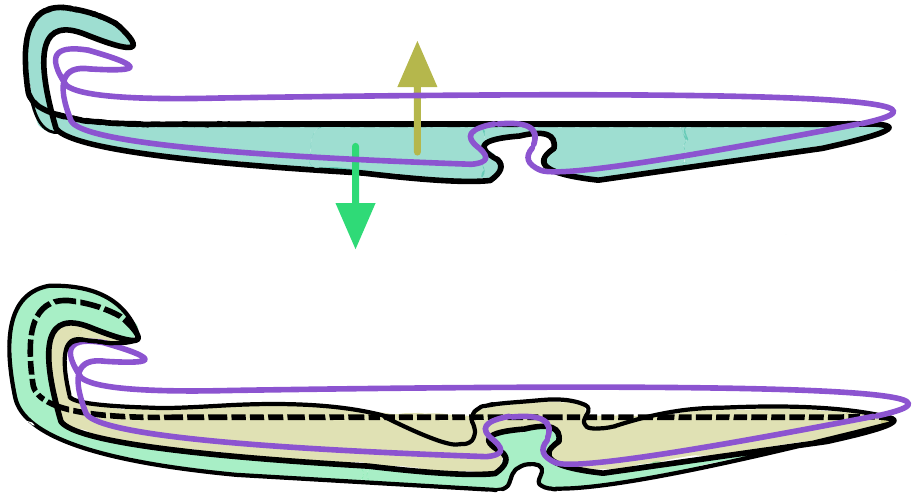}
    \caption{Obtaining a sphere by slightly pushing the interior of the Seifert surface in the negative and positive directions.}
    \label{fig22}
\end{figure}

Deforming the sphere slightly around its equator, we can have the original knot be in its interior while keeping the 0-pushoff outside. Thus, we get a splitting sphere.

Since an arbitrary 0-pushoff can be deformed to this 0-pushoff in the complement of the sphere, we are done with this implication.

$(\Longleftarrow)$Now suppose the link of a knot and its 0-pushoff is split. Since these are parallel copies, there is an annulus between the curves. Any separting 2-sphere intersects the annulus in some collection of curves. In particular, there should be a closed curve that runs around the annulus (observe the points at which the line connecting two parallel points on different knots intersects the 2-sphere). This is a closed curve on a sphere, so it should be unknotted, as it bounds an embedded disk on the sphere. Thus, each of the boundary knots of the annulus are also unknotted as they are parallel to the unknotted curve on the sphere. See figure \ref{fig23} for an illustration.
\end{proof}

\begin{figure}[h]
    \includegraphics[scale=0.35]{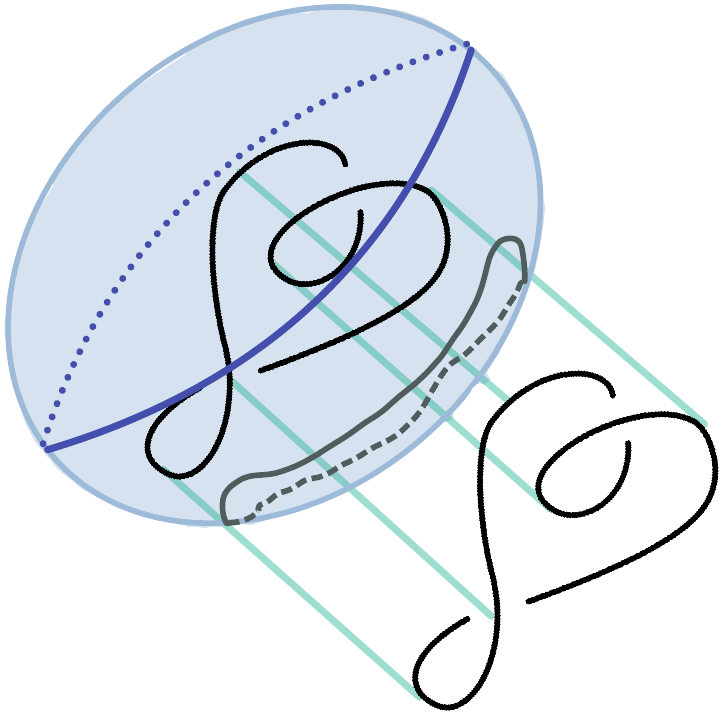}
    \caption{A splitting sphere gives an unknot that is parallel to our original knot, indicating that it is unknotted.}
    \label{fig23}
\end{figure}

For a knotted knot, it does not matter how the second copy is pushed off, the link is never split. For the unknot, however, it possible to get an unsplittable link with a twisted pushoff. That is why a 0-pushoff is necessary.

The split link approach avoids dealing with boundary conditions, since the surfaces sought after in this case are spheres. This also makes it simpler to construct a suitable triangulation, as it suffices to have a triangulation of $S^3$ with the link in its 1-skeleton. In fact, ideal triangulations \cite{Petronio} provide a much more efficient way to triangulate 3-manifolds with boundary by allowing deleted vertices.

This section presents an algorithm to determine if a two component link is split by searching for fundamental separating spheres. A corollary of this is unknotting. The algorithm follows the familiar scheme:
\begin{enumerate}
    \item Construct all fundamental solutions to the matching equations
    \item Check whether any fundamental solution $F$ has $\chi(F) = 2$.
    \item Check whether any fundamental sphere separates the components of the link \cite{joel}.
\end{enumerate}

The algorithm works if the following is true:

\begin{thm}
If $L$ is split, then $S^3 - L$ contains a fundamental splitting 2-sphere.
\end{thm}

\begin{rem}
    We assume that the link lies in the 1-skeleton. The proof for this theorem involves doing the normalization and fundamentalization steps, as described in Section \ref{norm}.
\end{rem}

The following lemma by Schubert is used in the proof:

\begin{lem}
(Schubert) Let $S$ represent a connected normal surface and suppose that $A$ and $B$ represent two normal surfaces such that $S = A + B$. If $A$ and $B$ are chosen to minimize $| A\cap B|$, then
\begin{enumerate}
    \item $A$ and $B$ each represent connected surfaces,
    \item no curve of intersection in $A \cap B$ is separating on both $A$ and $B$ \cite{schubert}.
\end{enumerate}
\end{lem}

\begin{proof}
(1) Suppose not. Then we can write $B = B_1 + B_2 + ... + B_k$, where each $B_i$ is a normal connected component of $B$, with $k \geq 2$. Since $S = A + B$ is a connected surface, $A$ must have non-zero intersection with each of the $B_i$'s. Perform Haken sums along the intersections of $A$ and $B_i$, for $i \neq 1$ until you get a normal surface $A' = A + B_2 + ... + B_k$. Then we have that $A' + B_1 = A + B = S$, but $|A' \cap B_1| < |A \cap B|$, a contradiction \cite{schubert}.

(2) Suppose a curve of $A \cap B$ is separating on both $A$ and $B$. Take the Haken sum along this curve. This curve separates both $A$ and $B$ into two pieces and the regular sum glues a piece of $A$ to a piece of $B$. After the regular sum, there are two possibly self-intersecting surfaces of a piece of $A$ glued to a piece of $B$. Take further Haken sums along the self-intersection curves of these two surfaces obtained from gluing a piece of $A$ to a piece of $B$ until both become normal. Denote these normal surfaces by $C$ and $D$. Then $C + D = A + B = S$, and $|C \cap D| < |A \cap B|$, a contradiction; see figure \ref{fig24} for an illustration \cite{schubert}.
\end{proof}

\begin{figure}[h]
    \includegraphics[scale=0.35]{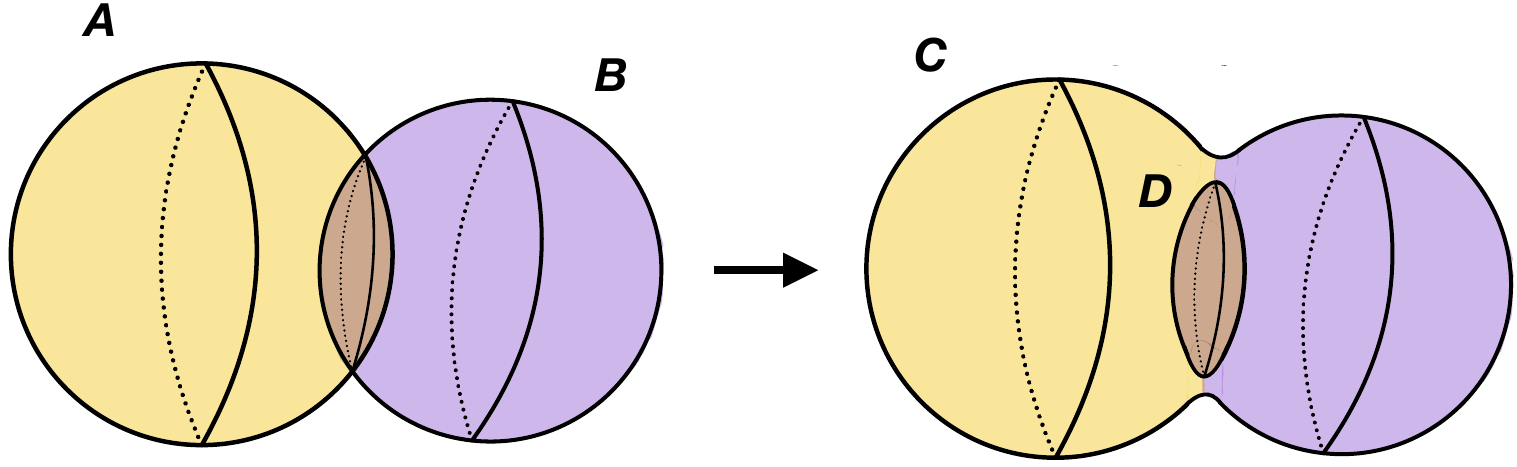}
    \caption{If an intersection curve is separating on both $A$ and $B$, one may obtain normal surfaces $C$, $D$ such that $C + D = A + B = S$ and $|C \cap D| < |A \cap B|$.}
    \label{fig24}
\end{figure}

\begin{proof} [Proof of Theorem 3.4]
\textbf{Step 1: Normalize.} Since $L$ is split, $S^3 - L$ contains a splitting 2-sphere. We must show that normalizing the 2-sphere does not affect its separating property. To do so, it suffices to show this for the normalization maneuvers. Note that although we work with a triangulation of $S^3$ and not $S^3 - L$, normalization of a sphere in $S^3 - L$ must remain in $S^3 - L$ since the process only reduces intersections with edges and $L$ lies in the 1-skeleton. Isotopy just moves the sphere around in $S^3 - L$, so does not change its separating property and boundary compression is only used for surfaces with boundary. Compression is not a problem too, because if a neck is compressed, one of the two new spheres must still separate the link components, as compression does not change intersections with the 1-skeleton (recall the proof of Theorem \ref{normsurf}). Since the compressing disk does not intersect the link, one knot must lie in one room of the sphere union the compressing disk and the other remains in the exterior. The resulting 2-sphere from the room splits the link \cite{joel}.

\textbf{Step 2: Fundamentalize.} Choose a normal splitting 2-sphere $S$ with smallest weight among all normal splitting 2-spheres (this is combinatorially the smallest such sphere). We claim that $S$ is fundamental. Suppose not; then $S = A + B$ for some non-empty normal surfaces $A$ and $B$. Assume also that $|A \cap B|$ is minimal, as in the statement of Schubert's lemma. By the additivity of the Euler characteristic, $\chi(S) = \chi(A) + \chi(B) = 2$. There are no closed surfaces of Euler characteristic 1 embedded in $S^3$, because the projective plane cannot be embedded in $S^3$. So, without loss of generality, $\chi(A) = 2$ and $\chi(B) = 0$: $S$ is the Haken sum a sphere $(A)$ and a torus $(B)$ \cite{joel} \cite{schubert}.

$A$ itself cannot be a splitting 2-sphere because the weight of $A$ is less than that of $S$ and we had chosen $S$ to have minimum weight among all normal splitting 2-spheres Thus, $A$ splits $S^3$ into two balls $D_1$ and $D_2$ in such a way that the link lies in one of these; assume it is in $D_1$. The torus, $B$, must be separating the knots.

Now, observe the intersection curves of $A$ and $B$. Since all closed curves are separating on $S^2$, by Schubert's lemma, the intersection curves on the torus must be non-separating. Thus, they are non-trivial curves in the fundamental group of the torus (Note that $A$ and $B$ cannot have one intersection curve, as normal surfaces do not intersect tangentially; this is because a regular sum cannot be performed for a tangential intersection of elementary disks. This can be proved by checking all the cases. As an example, see figure \ref{fig25}).

\begin{figure}[h]
    \includegraphics[scale=0.27]{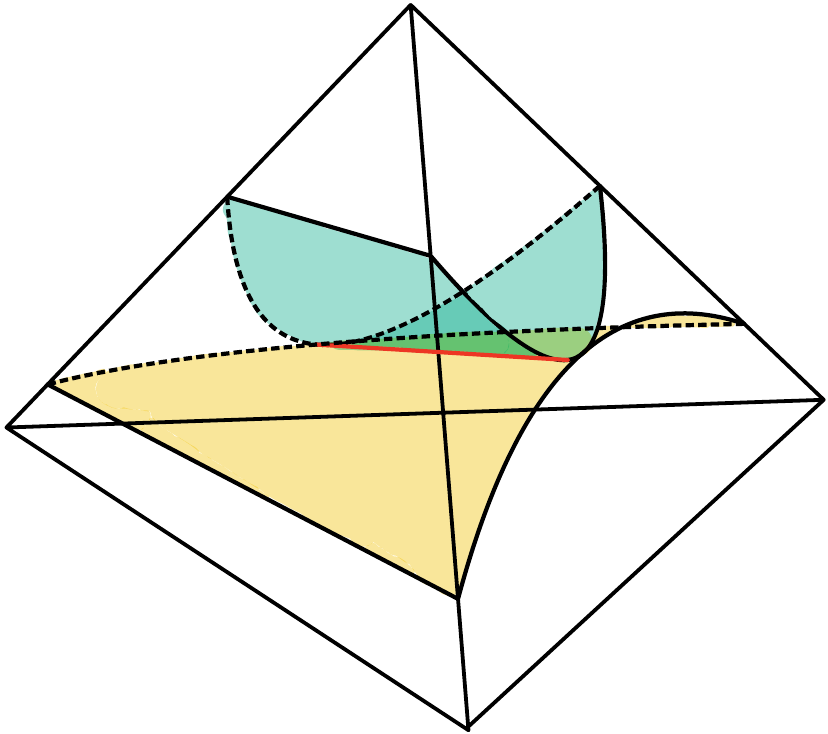}
    \caption{Two elementary disks that intersect tangentially. One cannot perform a Haken sum that gives elementary disks in the 3-simplex}
    \label{fig25}
\end{figure}

Thus, the torus $B$ is cut decomposed into a collection of open-ended cylinders by the intersection curves. These cylinders lie alternately in $D_1$ and $D_2$, as seen in figure \ref{fig26}. One of the knots must lie inside a single cylinder, as the cylinders are separated by disks on $A$, which cannot intersect the knot.

\begin{figure}[h]
    \includegraphics[scale=0.35]{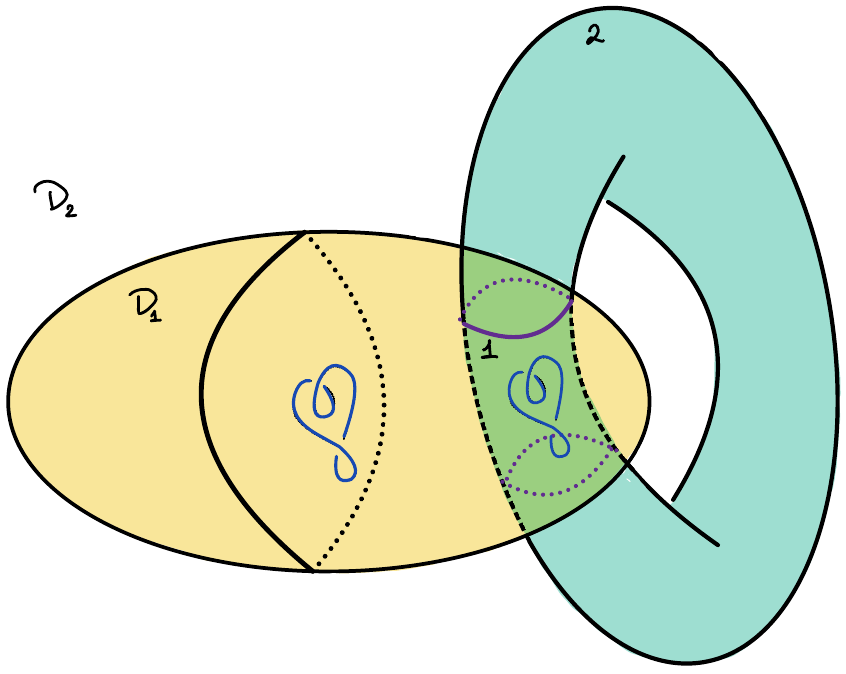}
    \caption{The sphere $A$ and the torus $B$ with the link illustrated in blue. The cylinders on the torus are numbered 1 and 2.}
    \label{fig26}
\end{figure}

Now, suppose that the two sides of the original splitting sphere, $S$, are colored black and white, respectively. Since $S$ splits the link, the two knots must be in regions with different colors. We also have that each cylinder on the torus must lie in a particularly colored region. Thus, we get that $D_1$ is decomposed by a cylinder on $B$ such that one knot lies in the cylinder and another outside.

By performing a combination of regular and irregular sums, we can glue the disks lying on $A$ at the boundaries of the cylinder to the ends of the cylinder. This gives another splitting sphere. Note that attaching boundary disks on $A$ must involve performing an irregular sum, as performing regular sums on all curves gives a connected surface. Since surfaces obtained by irregular sums can be normalized by decreasing weight, we can deform this to obtain a normal splitting 2-sphere of lesser weight that $S$, which is a contradiction \cite{joel} \cite{schubert}.

\begin{figure}[h]
    \includegraphics[scale=0.35]{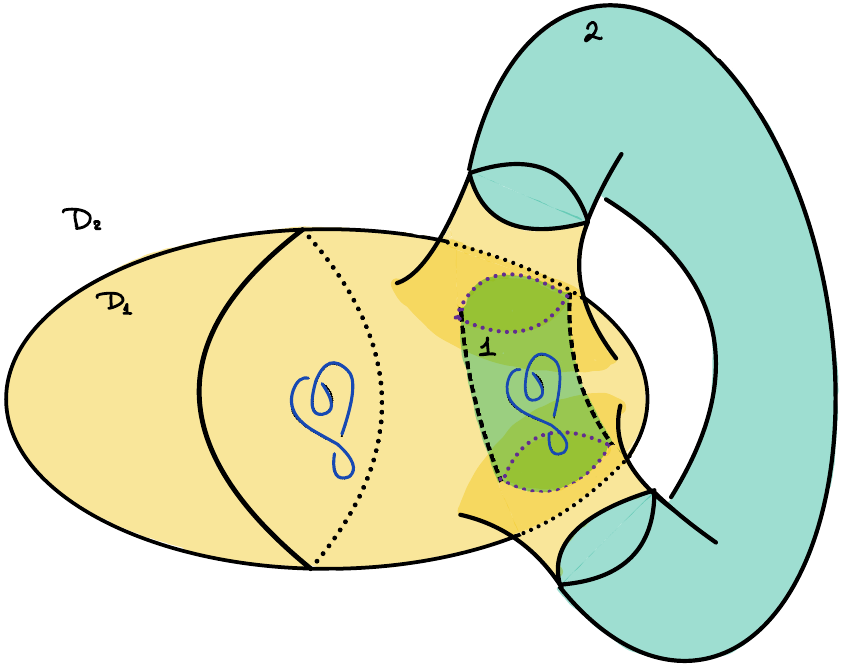}
    \caption{Obtaining a splitting 2-sphere from $A$ and $B$ through performing at least one irregular sum.}
    \label{fig27}
\end{figure}

\end{proof}

\section{Unknotting Example} \label{example}

In this section we use the split link algorithm to prove that the figure-eight knot is knotted. We construct a triangulation, solve the matching equation, search the fundamental surfaces for a splitting 2-sphere, and fail to find one.

\textbf{Step 1: Triangulation.} We triangulate the complement of a figure-eight knot with the 0-pushoff lying in the 1-skeleton of this triangulation, as a longitude on the boundary of a tubular neighborhood of the figure-eight knot. The longitude represents the trivial homology class and is therefore a 0-pushoff by definition.

To be able to solve the matching equations in a reasonable time, the triangulation must not contain too many tetrahedra. We use a combination of an inflation triangulation and ideal tetrahedra to obtain a triangulation for which the matching equations turn out to be very simple. Ideal tetrahedra are tetrahedra in ideal triangulations that have deleted vertices. Ideal triangulations are often used to obtain efficient triangulations of manifolds with boundary. See \cite{Petronio} for a detailed description of ideal triangulations. Inflation triangulations of ideal triangulations are explicit triangulations (without deleted vertices) of the manifold with boundary whose interior is triangulatied by the ideal triangulation. Inflation triangulations are introduced in \cite{jaco}. Given an ideal triangulation, an algorithm is given to obtain an inflation triangulation.

The complement of the figure eight knot has a simple ideal triangulation with two tetrahedra as seen in figure \ref{fig28}. We start with this. We then inflate this ideal triangulation to have an explicit triangulation of the 3-manifold with torus boundary. This inflation gives a triangulation of the complement of the figure-eight knot with 10 tetrahedra. This inflation is expained in detail in \cite{jaco} and it is conjectured that this is the most efficient triangulation of this manifold with torus boundary.

\begin{figure}[h]
    \includegraphics[scale=0.44]{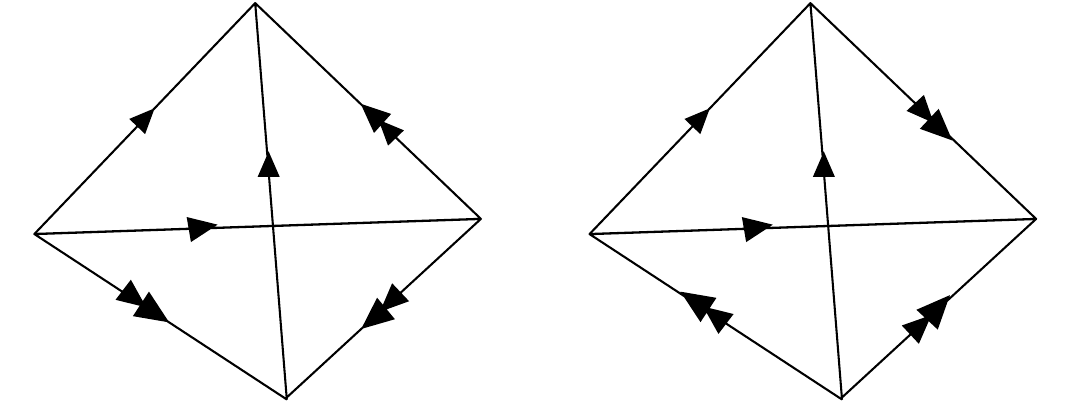}
    \caption{An ideal triangulation of the figure-eight knot complement in $S^3$. There is a unique way of identifying faces that agrees with the identifications of edges.}
    \label{fig28}
\end{figure}

The triangulation is presented below in Table \ref{tab1} with regina's notation. Table \ref{tab1} is directly from \cite{jaco}; we use the same names for the 3-simplices. The vertices are numbered 0 to 3. The notation $p(012)$ denotes the face of the 3-simplex $p$ with vertices $(0)$, $(1)$, $(2)$. The rows represent the tetrahedra, the columns represent the faces of the tetrahedra, and an entry represents the face with which the face in that column of the tetrahedron in that row is identified.

\begin{table}
\begin{center}

\begin{tabular}{lllll}
tet & $(012)$ & $(013)$ & $(023)$ & $(123)$ \\
$p$ & $3(320)$ & $\overline{4}(132)$ & $9(320)$ & $p'(320)$\\
$p'$ & $1(132)$ & $9(123)$ & $p(321)$ & $\overline{6}(032)$ \\
$1$ & $b_1^{*}(120)$ & $b_2^{*}(130)$ & $3(312)$ & $p'(021)$\\
$3$ & $c(130)$ & $\overline{6}(012)$ & $p(210)$ & $1(230)$\\
$\overline{4}$ & $c(021)$ & $b_1^{*}(130)$ & $\overline{6}(231)$ & $p(031)$\\
$\overline{6}$ & $3(013)$ & $c(023)$ & $p'(132)$ & $\overline{4}(302)$\\
$9$ & $b_2^{*}(230)$ & $c(132)$ & $p(320)$ & $p'(013)$\\
$c$ & $\overline{4}(021)$ & $3(201)$ & $\overline{6}(013)$ & $9(031)$\\
$b_1^{*}$ & $1(201)$ & $\overline{4}(301)$ & $b_2^{*}(021)$ & boundary\\
$b_2^{*}$ & $b_1^{*}(032)$ & $1(301)$ & $9(201)$ & boundary\\
\end{tabular}

\end{center}
\caption{Face identifications in the triangulation of the figure-eight knot complement.}
\label{tab1}
\end{table}

As seen in Table \ref{tab1}, the boundary torus is triangulated using two faces. This is the usual one vertex triangulation of the torus. Then, we add two tetrahedra to the other side of this boundary torus to make this into a triangulation of $S^3$. Removing the vertices of these tetrahedra that are not on the boundary and identified with each other, we acquire a triangulation of the figure-eight knot complement which is basically a subdivision of the ideal triangulation of the figure-eight knot complement. Thus, by using a combination of inflation triangulations and ideal triangulation, we get a triangulation with 12 tetrahedra to which the algorithm can be applied.

Table \ref{tab2} describes the new face identifications coming from the two added tetrahedra, $h_1$ and $h_2$.

\begin{table}
\begin{center}
\begin{tabular}{lllll}

tet & $(012)$ & $(013)$ & $(023)$ & $(123)$ \\
$b_1^{*}$ & $1(201)$ & $\overline{4}(301)$ & $b_2^{*}(021)$ & $h_1(321)$\\
$b_2^{*}$ & $b_1^{*}(032)$ & $1(301)$ & $9(201)$ & $h_2(123)$\\
$h_1$ & $h_2(012)$ & $h_2(032)$ & $h_2(031)$ & $b_1^{*}(321)$\\
$h_2$ & $h_1(012)$ & $h_1(032)$ & $h_1(031)$ & $b_2^{*}(123)$\\

\end{tabular}
\end{center}
\caption{The face identifications of the two added tetrahedra.}
\label{tab2}
\end{table}

\begin{rem}
    Note that our triangulation is not that of $S^3$ with two figure-eight knots embedded in the 1-skeleton. Instead, it is of the complement of the figure-eight knot with its 0-pushoff in the 1-skeleton. It is clear that the above argument works in this case too.
\end{rem}

Now we have to specify the 0-pushoff knot.
\begin{claim}
    The knot specified by the edge $b_1^{*}(13)$ is a 0-pushoff.
\end{claim}

\begin{proof}
    We show that this edge is zero in homology. Our strategy is to search through the edge identifications in Table \ref{tab1} and show that $b_1^{*}(13)$ is trivial in the homology of the 2-skeleton. We by no means assert that our search is the most efficient. Let $a$ be the homology class of $b_1^{*}(13)$, i.e. $[b_1^{*}(13)] = a$.

    First let's write the edge identifications for $b_1^{*}(13)$ from Table \ref{tab1}:
    \begin{equation}
        b_1^{*}(13) \sim \overline{4}(01) \sim c(02) \sim \overline{6}(01) \sim 3(01) \sim c(13) \sim 9(01) \sim b_2^{*}(23).
    \label{eq1}
    \end{equation}

    Looking at the face $9(013)$, we see that the knot $9(03) + 9(31)$ is isotopic to our knot, which is given by $9(01)$ as seen in the identifications above. Thus, $[9(03) + 9(31)] = a$ -- it represents the same class in homology.

    \begin{figure}[h]
    \includegraphics[scale=0.2]{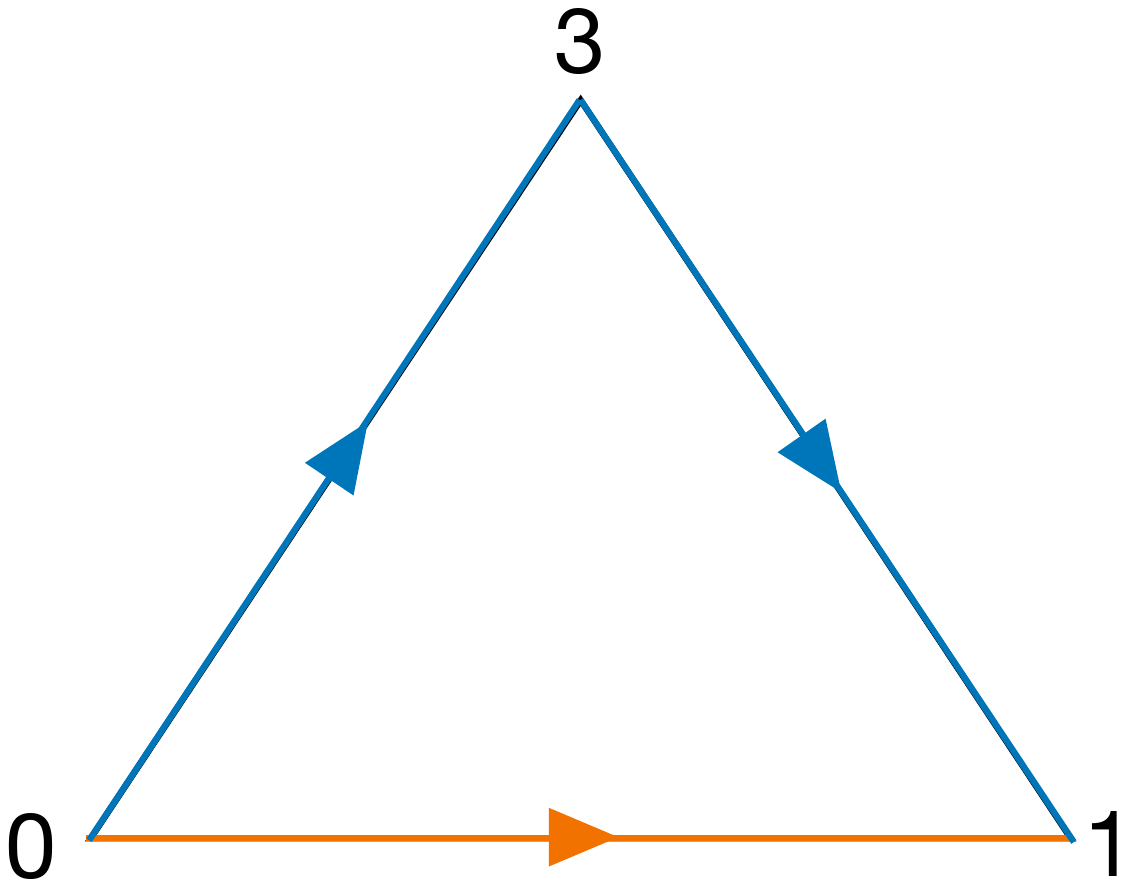}
    \caption{The face $9(013)$. By homotopy, we see that $9(03) + 9(31) \sim 9(01)$.}
    \label{fig29}
    \end{figure}

    By Table \ref{tab1}, we further see that $9(03) \sim p(30)$, $9(31) \sim p'(30)$, and $p'(30) \sim p(13)$. Thus, $9(03) + 9(31) \sim p(30) + p'(30) \sim p(30) + p(13)$.

    Now, looking at the face $p(013)$, we get that $p(30) + p(13)$ represents the same homology class as $p(10)$, i.e. $[p(10)] = a$.

    Again, following edge identifications from Table \ref{tab1}, we get that
    \begin{equation}
        p(10) \sim 3(32) \sim 1(03) \sim b_2^{*}(10) \sim b_1^{*}(30) \sim \overline{4}(13).
    \label{eq2}
    \end{equation}

    Looking at the face $\overline{4}(013)$, we have that $[\overline{4}(01)] = [\overline{4}(13)]$, as seen from \ref{eq1} and \ref{eq2}. This implies that $[\overline{4}(03)] = 2a$; see figure \ref{fig30}.

    \begin{figure}[h]
    \includegraphics[scale=0.24]{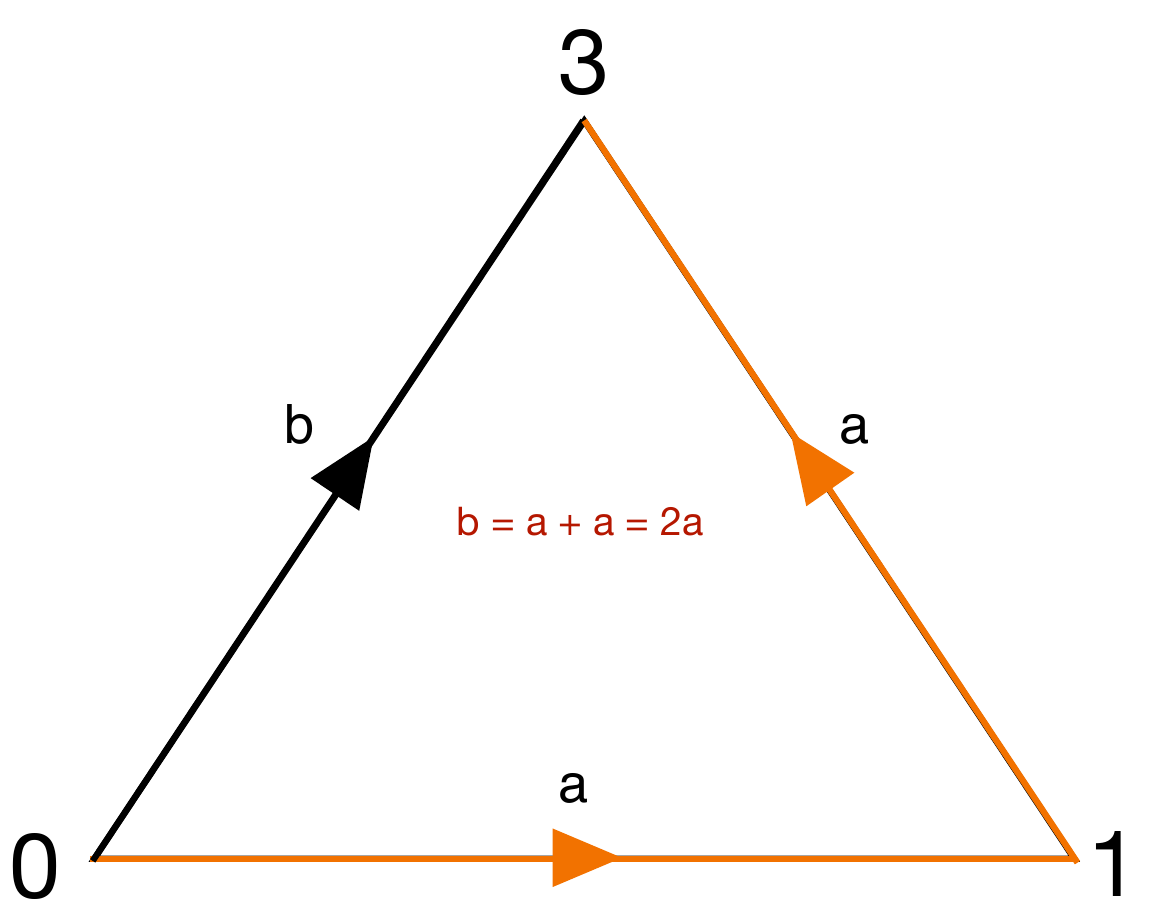}
    \caption{}
    \label{fig30}
\end{figure}

    Once again, we follow the edge identifications from Table \ref{tab1}:
    \begin{equation}
        \overline{4}(03) \sim b_1^{*}(10) \sim 1(02) \sim 3(31).
    \label{eq3}
    \end{equation}

    From \ref{eq1}, \ref{eq2} and \ref{eq3}, we get that $[3(01)] = [3(32)] = a$ and $[3(31)] = 2a$, respectively. As we did previously, these imply that $[3(21)] = [3(30)] = a$ which in turn imply that $[3(20)] = 0$ (implications follow from looking at the appropriate faces for which the homology classes of two of the three edges are known).

    From Table \ref{tab1}, we have that $[p(12)] = [3(20)] = 0$ and $[p(02)] = [3(30)] = 0$. Recall from before that $[p(10)] = a$. With this information, looking at the face $p(012)$, we see that $2a = 0$, implying that $a = 0$, because the homology on a knot exterior is infinite cyclic \cite{rolfsen}. This proves that $b_1^{*}(13)$ is a 0-pushoff.
\end{proof}

\textbf{Step 2: Matching equations.} Each tetrahedron contributes 7 variables to the matching equations. The variables for the tetrahedron $p$ are the entries of its integer vector, denoted as $p = (p_1, p_2, p_3, p_4, p_5, p_6, p_7)$. The variables $p_1, p_2, p_3, p_4$ count the normal triangles associated with the vertices $(0), (1), (2), (3)$, respectively; the variables $p_5, p_6, p_7$ count the normal quadrilaterals that separate the pairs of vertices $(0)(1)$ and $(1)(2)$, $(0)(2)$ and $(1)(3)$, $(0)(3)$ and $(1)(2)$, respectively. Only one of $p_5, p_6, p_7$ is non-zero. Now, using the face identifications, we write down the matching equations:

\begin{multicols}{2}
\begin{enumerate}
    \item $p(012) \sim 3(320):$
    \begin{enumerate} [label=\roman*.]
        \item $p_1 + p_7 = 3_4 + 3_6$
        \item $p_2 + p_6 = 3_3 + 3_7$
        \item $p_3 + p_5 = 3_1 + 3_5$
    \end{enumerate}
    \vspace{2mm}
    
    \item $p(013) \sim \overline{4}(132):$
    \begin{enumerate} [label=\roman*.]
        \item $p_1 + p_6 = \overline{4}_2 + \overline{4}_5$
        \item $p_2 + p_7 = \overline{4}_4 + \overline{4}_7$
        \item $p_4 + p_5 = \overline{4}_3 + \overline{4}_6$
    \end{enumerate}
    \vspace{2mm}
    
    \item $p(023) \sim 9(320):$
    \begin{enumerate} [label=\roman*.]
        \item $p_1 + p_5 = 9_4 + 9_6$
        \item $p_3 + p_7 = 9_3 + 9_7$
        \item $p_4 + p_6 = 9_1 + 9_5$
    \end{enumerate}
    \vspace{2mm}
    
    \item $p(123) \sim p'(320):$
    \begin{enumerate} [label=\roman*.]
        \item $p_2 + p_5 = p'_4 + p'_6$
        \item $p_3 + p_6 = p'_3 + p'_7$
        \item $p_4 + p_7 = p'_1 + p'_5$
    \end{enumerate}
    \vspace{2mm}
    
    \item $p'(012) \sim 1(132):$
    \begin{enumerate} [label=\roman*.]
        \item $p'_1 + p'_7 = 1_2 + 1_5$
        \item $p'_2 + p'_6 = 1_4 + 1_7$
        \item $p'_3 + p'_5 = 1_3 + 1_6$
    \end{enumerate}
    \vspace{2mm}
    
    \item $p'(013) \sim 9(123):$
    \begin{enumerate} [label=\roman*.]
        \item $p'_1 + p'_6 = 9_2 + 9_5$
        \item $p'_2 + p'_7 = 9_3 + 9_6$
        \item $p'_4 + p'_5 = 9_4 + 9_7$
    \end{enumerate}
    \vspace{2mm}
    
    \item $p'(123) \sim \overline{6}(032):$
    \begin{enumerate} [label=\roman*.]
        \item $p_2 + p_5 = \overline{6}_1 + \overline{6}_5$
        \item $p_3 + p_6 = \overline{6}_4 + \overline{6}_6$
        \item $p_4 + p_7 = \overline{6}_3 + \overline{6}_7$
    \end{enumerate}
    \vspace{2mm}
    
    \item $1(012) \sim b_1^{*}(120):$
    \begin{enumerate} [label=\roman*.]
        \item $1_1 + 1_7 = b_1^{*}{}_2 + b_1^{*}{}_6$
        \item $1_2 + 1_6 = b_1^{*}{}_3 + b_1^{*}{}_5$
        \item $1_3 + 1_5 = b_1^{*}{}_1 + b_1^{*}{}_7$
    \end{enumerate}
    \vspace{2mm}
    
    \item $1(013) \sim b_2^{*}(130):$
    \begin{enumerate} [label=\roman*.]
        \item $1_1 + 1_6 = b_2^{*}{}_2 + b_2^{*}{}_7$
        \item $1_2 + 1_7 = b_2^{*}{}_4 + b_2^{*}{}_5$
        \item $1_4 + 1_5 = b_2^{*}{}_1 + b_2^{*}{}_6$
    \end{enumerate}
    \vspace{2mm}
    
    \item $1(023) \sim 3(312):$
    \begin{enumerate} [label=\roman*.]
        \item $1_1 + 1_5 = 3_4 + 3_7$
        \item $1_3 + 1_7 = 3_2 + 3_5$
        \item $1_4 + 1_6 = 3_3 + 3_6$
    \end{enumerate}
    \vspace{2mm}
    
    \item $3(012) \sim c(130):$
    \begin{enumerate} [label=\roman*.]
        \item $3_1 + 3_7 = c_2 + c_7$
        \item $3_2 + 3_6 = c_4 + c_5$
        \item $3_3 + 3_5 = c_1 + c_6$
    \end{enumerate}
    \vspace{2mm}
    
    \item $3(013) \sim \overline{6}(012):$
    \begin{enumerate} [label=\roman*.]
        \item $3_1 + 3_6 = \overline{6}_1 + \overline{6}_7$
        \item $3_2 + 3_7 = \overline{6}_2 + \overline{6}_6$
        \item $3_4 + 3_5 = \overline{6}_3 + \overline{6}_5$
    \end{enumerate}
    \vspace{2mm}
    
    \item $\overline{4}(012) \sim c(021):$
    \begin{enumerate} [label=\roman*.]
        \item $\overline{4}_1 + \overline{4}_7 = c_1 + c_7$
        \item $\overline{4}_2 + \overline{4}_6 = c_3 + c_5$
        \item $\overline{4}_3 + \overline{4}_5 = c_2 + c_6$
    \end{enumerate}
    \vspace{2mm}
    
    \item $\overline{4}(013) \sim b_1^{*}(130):$
    \begin{enumerate} [label=\roman*.]
        \item $\overline{4}_1 + \overline{4}_6 = b_1^{*}{}_2 + b_1^{*}{}_7$
        \item $\overline{4}_2 + \overline{4}_7 = b_1^{*}{}_4 + b_1^{*}{}_5$
        \item $\overline{4}_4 + \overline{4}_5 = b_1^{*}{}_1 + b_1^{*}{}_6$
    \end{enumerate}
    \vspace{2mm}
    
    \item $\overline{4}(023) \sim \overline{6}(231):$
    \begin{enumerate} [label=\roman*.]
        \item $\overline{4}_1 + \overline{4}_5 = \overline{6}_3 + \overline{6}_6$
        \item $\overline{4}_3 + \overline{4}_7 = \overline{6}_4 + \overline{6}_7$
        \item $\overline{4}_4 + \overline{4}_6 = \overline{6}_2 + \overline{6}_5$
    \end{enumerate}
    \vspace{2mm}
    
    \item $\overline{6}(013) \sim c(023):$
    \begin{enumerate} [label=\roman*.]
        \item $\overline{6}_1 + \overline{6}_6 = c_1 + c_5$
        \item $\overline{6}_2 + \overline{6}_7 = c_3 + c_7$
        \item $\overline{6}_4 + \overline{6}_5 = c_4 + c_6$
    \end{enumerate}
    \vspace{2mm}
    
    \item $9(012) \sim b_2^{*}(230):$
    \begin{enumerate} [label=\roman*.]
        \item $9_1 + 9_7 = b_2^{*}{}_3 + b_2^{*}{}_7$
        \item $9_2 + 9_6 = b_2^{*}{}_4 + b_2^{*}{}_6$
        \item $9_3 + 9_5 = b_2^{*}{}_1 + b_2^{*}{}_5$
    \end{enumerate}
    \vspace{2mm}
    
    \item $9(013) \sim c(132):$
    \begin{enumerate} [label=\roman*.]
        \item $9_1 + 9_6 = c_2 + c_5$
        \item $9_2 + 9_7 = c_4 + c_7$
        \item $9_4 + 9_5 = c_3 + c_6$
    \end{enumerate}
    \vspace{2mm}
    
    \item $b_1^{*}(023) \sim b_2^{*}(021):$
    \begin{enumerate} [label=\roman*.]
        \item $b_1^{*}{}_1 + b_1^{*}{}_5 = b_2^{*}{}_1 + b_2^{*}{}_7$
        \item $b_1^{*}{}_3 + b_1^{*}{}_7 = b_2^{*}{}_3 + b_2^{*}{}_5$
        \item $b_1^{*}{}_4 + b_1^{*}{}_6 = b_2^{*}{}_2 + b_2^{*}{}_6$
    \end{enumerate}
    \vspace{2mm}
    
    \item $b_1^{*}(123) \sim h_1(321):$
    \begin{enumerate} [label=\roman*.]
        \item $b_1^{*}{}_2 + b_1^{*}{}_5 = h_1{}_4 + h_1{}_7$
        \item $b_1^{*}{}_3 + b_1^{*}{}_6 = h_1{}_3 + h_1{}_6$
        \item $b_1^{*}{}_4 + b_1^{*}{}_7 = h_1{}_2 + h_1{}_5$
    \end{enumerate}
    \vspace{2mm}
    
    \item $b_2^{*}(123) \sim h_2(123):$
    \begin{enumerate} [label=\roman*.]
        \item $b_2^{*}{}_2 + b_2^{*}{}_5 = h_2{}_2 + h_2{}_5$
        \item $b_2^{*}{}_3 + b_2^{*}{}_6 = h_2{}_3 + h_2{}_6$
        \item $b_2^{*}{}_4 + b_2^{*}{}_7 = h_2{}_4 + h_2{}_7$
    \end{enumerate}
    \vspace{2mm}
    
    \item $h_1(130) \sim h_2(320):$
    \begin{enumerate} [label=\roman*.]
        \item $h_1{}_2 + h_1{}_7 = h_2{}_4 + h_2{}_6$
        \item $h_1{}_4 + h_1{}_5 = h_2{}_3 + h_2{}_7$
        \item $h_1{}_1 + h_1{}_6 = h_2{}_1 + h_2{}_5$
    \end{enumerate}
    \vspace{2mm}
    
    \item $h_1(023) \sim h_2(031):$
    \begin{enumerate} [label=\roman*.]
        \item $h_1{}_1 + h_1{}_5 = h_2{}_1 + h_2{}_6$
        \item $h_1{}_3 + h_1{}_7 = h_2{}_4 + h_2{}_5$
        \item $h_1{}_4 + h_1{}_6 = h_2{}_2 + h_2{}_7$
    \end{enumerate}
    \vspace{2mm}
    
    \item $h_1(012) \sim h_2(012):$
    \begin{enumerate} [label=\roman*.]
        \item $h_1{}_1 + h_1{}_7 = h_2{}_1 + h_2{}_7$
        \item $h_1{}_2 + h_1{}_6 = h_2{}_2 + h_2{}_6$
        \item $h_1{}_3 + h_1{}_5 = h_2{}_3 + h_2{}_5$
    \end{enumerate}
    \vspace{2mm}
    
\end{enumerate}
\end{multicols}

This is a large system of linear equations and solving it requires considerable computer time. However, there are certain restriction on the values of the variables that simplifies the system. These restrictions arise from two requirements: (1) the normal surface cannot intersect the knot and its 0-pushoff, (2) there can only be one type of quadrilateral in each 3-simplex.

We first examine (1). One of the knots is represented by an identified vertex of $h_1$ and $h_2$. Normal surfaces do not intersect the 1-skeleton, so this knot does not place any restrictions. The second knot is $b_1^{*}(13)$ in the 1-skeleton of the boundary torus of the inflation triangulation. Thus the variables corresponding to any elementary disks intersecting the edges in \ref{eq1} must be set to zero. We therefore have the following restrictions:

\vspace{2mm}

\begin{flushleft}
    $b_1^{*}{}_2 = b_1^{*}{}_4 = b_1^{*}{}_5 = b_1^{*}{}_7 = 0$
    
    $h_1{}_2 = h_1{}_4 = h_1{}_5 = h_1{}_7 = 0$
    
    $h_2{}_3 = h_2{}_4 = h_2{}_6 = h_2{}_7 = 0$
    
    $b_2^{*}{}_3 = b_2^{*}{}_4 = b_2^{*}{}_6 = b_2^{*}{}_7 = 0$
    
    $9_1 = 9_2 = 9_6 = 9_7 = 0$
    
    $c_1 = c_2 = c_3 = c_4 = c_5 = c_7 = 0$
    
    $3_1 = 3_2 = 3_6 = 3_7 = 0$
    
    $\overline{6}_1 = \overline{6}_2 = \overline{6}_6 = \overline{6}_7 = 0$
    
    $\overline{4}_1 = \overline{4}_2 = \overline{4}_6 = \overline{4}_7 = 0$
\end{flushleft}

\vspace{2mm}

\begin{flushleft}
The matching equations now simplify to the following system:
\end{flushleft}

\vspace{2mm}

\begin{tabular}{lll}
$1_1 + 1_5 = p_1 + p_7$ & $p_2 + p_5 = p'_4 + p'_6$ & $p'_1 + p'_7 = 1_2 + 1_5$ \\
$1_3 + 1_7 = p_3 + p_5$ & $p_3 + p_6 = p'_3 + p'_7$ & $p'_2 + p'_6 = 1_3 + 1_7$ \\
$1_3 + 1_7 = p_2 + p_6$ & $p_4 + p_7 = p'_1 + p'_5$ & $p'_3 + p'_5 = 1_3 + 1_7$ \\
\end{tabular}

\vspace{2mm}

\begin{tabular}{lll}
$p'_1 + p'_6 = p_4 + p_6$ & $p_1 + p_6 = p'_4 + p'_7$ & $p_4 + p_1 = p_2 + p_3$ \\
$p'_2 + p'_7 = p_3 + p_7$ & $p_4 + p_5 = p'_3 + p'_6$ & $p_1 + p_6 = p_2 + p_7$ \\
$p'_4 + p'_5 = p_1 + p_5$ & $p_2 + p_7 = p'_2 + p'_5$ & $1_2 + 1_1 + 2.1_7 = 2.h_1{}_3$ \\
\end{tabular}

\vspace{2mm}

\begin{tabular}{lll}
$h_1{}_1 = h_2{}_1$ & & \\
\end{tabular}

\vspace{2mm}

The remaining variables depend on these variables. Thus, the number of variables is significantly decreased and the system is much simpler to solve.

Now, we impose the condition that there can be only one type of quadrilateral in each 3-simplex to further simplify the system. Due to the previous restrictions, the only tetrahedra for which a choice is necessary are $p$ and $p'$ (although $1_7$ and $1_6$ are not explicitly set to zero above in the above restrictions, plugging those restrictions to the matching equations implies that $1_7 = 1_6$, implying that both are zero). Thus, there are $3^2 = 9$ cases to consider. However, it turns out that all these choices yield the same equations, i.e. there are no normal surfaces missing the knots with quadrilaterals in $p$ or $p'$.

The system we get is the following:

\vspace{2mm}

\begin{tabular}{lll}
$1_1 + 1_5 = p_1$ & & \\
$h_1{}_1 = h_2{}_1$ & & \\
\end{tabular}

\vspace{2mm}

All other variables are dependent variables. The fundamental solutions to this system are seen to be
\begin{enumerate}
    \item $1_1 = 0$, $1_5 = 1$, $p_1 = 1$, $h_1{}_1 = 0$
    \item $1_1 = 1$, $1_5 = 0$, $p_1 = 1$, $h_1{}_1 = 0$
    \item $1_1 = 0$, $1_5 = 0$, $p_1 = 0$, $h_1{}_1 = 1$.
\end{enumerate}

Finally, the integer vectors corresponding to the fundamental solutions are
\begin{enumerate}
    \item $p = (1,1,1,1,0,0,0)$, $p' = (1,1,1,1,0,0,0)$, $1 = (0,0,1,1,1,0,0)$, $\overline{4} = (0,0,1,1,1,0,0)$, $3 = (0,0,1,1,1,0,0)$, $\overline{6} = (0,0,1,1,1,0,0)$, $9 = (0,0,1,1,1,0,0)$, $c = (0,0,0,0,0,2,0)$, $b_1^{*} = (2,0,0,0,0,0,0)$, $b_2^{*} = (2,0,0,0,0,0,0)$, $h_1 = (0,0,0,0,0,0,0)$, $h_2 = (0,0,0,0,0,0,0)$
    \item $p = (1,1,1,1,0,0,0)$, $p' = (1,1,1,1,0,0,0)$, $1 = (1,1,1,1,0,0,0)$, $\overline{4} = (0,0,1,1,1,0,0)$, $3 = (0,0,1,1,1,0,0)$, $\overline{6} = (0,0,1,1,1,0,0)$, $9 = (0,0,1,1,1,0,0)$, $c = (0,0,0,0,0,2,0)$, $b_1^{*} = (1,0,1,0,0,1,0)$, $b_2^{*} = (1,1,0,0,1,0,0)$, $h_1 = (0,0,1,0,0,1,0)$, $h_2 = (0,1,0,0,1,0,0)$
    \item $p = (0,0,0,0,0,0,0)$, $p' = (0,0,0,0,0,0,0)$, $1 = (0,0,0,0,0,0,0)$, $\overline{4} = (0,0,0,0,0,0,0)$, $3 = (0,0,0,0,0,0,0)$, $\overline{6} = (0,0,0,0,0,0,0)$, $9 = (0,0,0,0,0,0,0)$, $c = (0,0,0,0,0,0,0)$, $b_1^{*} = (0,0,0,0,0,0,0)$, $b_2^{*} = (0,0,0,0,0,0,0)$, $h_1 = (1,0,0,0,0,0,0)$, $h_2 = (1,0,0,0,0,0,0)$.
\end{enumerate}

\vspace{2mm}

\textbf{Step 3: Search the fundamental surfaces.}

(1) is clearly not a separating sphere, because there are no elementary disks in $h_1$ and $h_2$, implying that the surface does not separate the knots.

(2) is not a sphere, because although it separates the knots, it has Euler characteristic 0; it is a normal torus.

(3) is also a torus, which is parallel to the boundary of the inflation triangulation.

This proves that the figure-eight knot is knotted.

\bibliographystyle{plain}
\bibliography{refs}

\end{document}